\documentclass[12pt, reqno]{amsart}
\usepackage[utf8]{inputenc}	
\usepackage{amssymb,amscd}
\usepackage{multirow,booktabs}
\usepackage[table]{xcolor}
\usepackage{fullpage}
\usepackage{lastpage}
\usepackage{fancyhdr}
\usepackage{mathrsfs}
\usepackage{wrapfig}
\usepackage{graphicx}
\usepackage{setspace}
\usepackage{calc}
\usepackage{bm}
\usepackage{subfigure}
\usepackage{tikz}
\usetikzlibrary{arrows,
                quotes}
\tikzset{
every edge quotes/.append style = {font=\footnotesize},
       every edge/.append style = {> = stealth, 
                                   draw=red, semithick},
                  vertex/.style = {shape=circle, draw, minimum size=1.5em, inner sep=2pt}
        } 
\usepackage{clock}
\usepackage{enumitem}
\usepackage{multicol}
\usepackage[colorlinks,
            linkcolor=blue,
            anchorcolor=blue,
            citecolor=red,
            ]{hyperref}
\usepackage{cancel}
\usepackage[all]{xy}
\usepackage[margin=3cm]{geometry}

\usepackage{empheq}
\usepackage{framed}
\usepackage{dsfont}
\usepackage{xcolor}
\theoremstyle{plain}
\newtheorem{theorem}{Theorem}[section]
\newtheorem{lemma}[theorem]{Lemma}
\newtheorem{proposition}[theorem]{Proposition}
\newtheorem{corollary}[theorem]{Corollary}

\theoremstyle{remark}
\newtheorem{definition}{Definition}
\newtheorem{example}{Example}[section]
\newtheorem{remark}{Remark}[section]

\usepackage{indentfirst}
\begin{document}
\setcounter{section}{0}

\thispagestyle{empty}

\newcommand{\QQ}{\mathbb{Q}}
\newcommand{\R}{\mathbb{R}}
\newcommand{\Z}{\mathbb{Z}}
\newcommand{\N}{\mathbb{N}}
\newcommand{\RR}{\mathbb{R}}
\newcommand{\ZZ}{\mathbb{Z}}
\newcommand{\NN}{\mathbb{N}}
\newcommand{\Nor}{\mathscr{N}}
\newcommand{\CC}{\mathbb{C}}
\newcommand{\HH}{\mathbb{H}}
\newcommand{\EE}{\mathbb{E}}
\newcommand{\Var}{\operatorname{Var}}
\newcommand{\PP}{\mathbb{P}}
\newcommand{\Rd}{\mathbb{R}^d}
\newcommand{\Rn}{\mathbb{R}^n}
\newcommand{\XX}{\mathcal{X}}
\newcommand{\YY}{\mathcal{Y}}
\newcommand{\MM}{\FF}
\newcommand{\BHH}{\overline{\mathbb{H}}}
\newcommand{\XB}{( \mathcal{X},\mathscr{B} )}
\newcommand{\BB}{\mathscr{B}}
\newcommand{\system}{(\Omega,\mathcal{F},\mu,T)}
\newcommand{\FF}{\mathcal{F}}
\newcommand{\GG}{\mathcal{G}}
\newcommand{\MBS}{(\Omega,\mathcal{F})}
\newcommand{\MBSE}{(E,\mathscr{E})}
\newcommand{\MS}{(\Omega,\mathcal{F},\mu)}
\newcommand{\PS}{(\Omega,\mathcal{F},\mathbb{P})}
\newcommand{\LDP}{LDP(\mu_n, r_n, I)}
\newcommand{\Def}{\overset{\text{def}}{=}}
\newcommand{\Series}[2]{#1_1,\cdots,#1_#2}
\newcommand{\independent}{\perp\mkern-9.5mu\perp}
\def\avint{\mathop{\,\rlap{-}\!\!\int\!\!\llap{-}}\nolimits}

\author[Shuo Qin]{Shuo Qin}
\address[Shuo Qin]{NYU-ECNU Institute of Mathematical Sciences at NYU Shanghai and Courant Institute of Mathematical Sciences \& Beijing Institute of Mathematical Sciences and Applications}
\email{qinshuo@bimsa.cn}

\title{Interacting urn models with strong reinforcement}
\date{}

  \begin{abstract}
 For the interacting urn model with polynomial reinforcement, it has been conjectured in \cite{launay2011interacting} that almost surely one color monopolizes all the urns if the interaction parameter $p>0$. We disprove the conjecture.  

    For the case $p=1$, we give a sufficient condition for monopoly, which improves a result obtained by Launay in \cite{launay2012urns}.
  \end{abstract}
\maketitle

\section{General introduction}

\subsection{Definition of the model}

Reinforced processes provide a rich framework for modeling and analyzing systems in physics, economics, and social sciences, where history plays a crucial role
in shaping future dynamics. We refer to \cite{MR2282181} for a survey of various models of random processes with reinforcement and their applications, a basic model of which is the well-known P\'olya urn model. A generalized P\'olya urn model is defined as follows. 

Let $\{W(n)\}_{n\geq 1}$ be a positive sequence. Given an urn of black and red balls, let $B_n$ and $R_n$ denote the number of black and red balls in the urn at time $n \in \NN$, respectively. We assume that $B_0$ and $R_0$ are positive integers. At each time step $n\geq 1$, we draw a ball from the urn and then return the ball to the urn along with another ball of the same color. The probability of drawing a ball of a certain color from the urn is proportional to $W(k)$ where $k$ is the number of balls of this color, that is,
$$
\PP(B_{n+1}=B_n+1|(B_i)_{0\leq i\leq n},(R_i)_{\leq i\leq n})=\frac{W(B_n)}{W(B_n)+W(R_{n})}, \quad n\in \NN.
$$
Notice that the classical P\'olya urn corresponds to the case $W(n)=n$.

In \cite{Drinea_Frieze_Mitzenmacher_2002, MR2376425}, this model was called a balls-in-bins process with feedback where the authors were motivated by economic problems of competition and the sequence $\{W(n)\}_{n\geq 1}$ was called the feedback function. It is also referred to as an ordinal dependent P\'olya urn \cite{MR2282181} since it is equal in law to the following process: At each time step, we draw a ball from the urn randomly uniformly with replacement, and add $W(n+1)-W(n)$ red balls, resp. $W(n+1)-W(n)$ black balls, if it is the $n$-th time a red ball, resp. a black ball is drawn.

We denote by $\mathcal{D}$ the event that eventually only balls of one color are added to the urn. Using Rubin’s exponential embedding, Davis \cite{MR1030727} proved that 
\begin{equation}
    \label{Davissumcond}
    \PP(\mathcal{D})= \left \{ \begin{aligned} &1,  && \text{if } \sum_{n=1}^{\infty}\frac{1}{W(n)}<\infty, \\ & 0, && \text{if } \sum_{n=1}^{\infty}\frac{1}{W(n)}=\infty. \end{aligned} \right.
\end{equation}

Recently, there is a growing interest in the study of systems of interacting urns, see e.g. \cite{MR2493995, MR3346459, MR3452818, MR4547356, MR3217785, hu2011reinforcement, MR4546116, mirebrahimi2019interacting, MR3581253}. We study a model of (strongly) reinforced interacting urns introduced by Launay \cite{launay2011interacting}, which can be described as follows.

The model has $d$ urns containing black and red balls. Imagine that there are barriers separating different urns. At each step, for the i-th urn, $i=1,2,\cdots, d$, 
\begin{enumerate}  
    \item with probability $p\in [0,1]$, (all the barriers are removed, and) a ball is drawn from a combined pool of all urns with replacement, see e.g. Figure \ref{fig2urnp} for the case $d=2$;  
    \item with probability $1-p$, (the barriers are kept, and) a ball is drawn from the i-th urn with replacement, see e.g. Figure \ref{fig2urn1p};
    \item The probability of drawing a ball of a certain color is proportional to $W$(\#the number of balls of that color), as in the ordinal-dependent P\'olya's urn;
    \item In either case, we add another ball of the same color as the drawn ball to the i-th urn.
\end{enumerate}
We do the above procedure simultaneously and independently for each urn.

More precisely, let $B_n(i)$ and $R_n(i)$ denote the number of black and red balls in the i-th urn at time $n$, respectively. Then, $B_n^{*}:=\sum_{i=1}^d B_n(i)$, resp. $R_n^{*}:=\sum_{i=1}^d R_n(i)$, is the total number of black balls, resp. red balls, in the system at time $n$. Write $R_n:=(R_n(i))_{1\leq i \leq d}$ and $B_n:=(B_n(i))_{1\leq i \leq d}$. The initial composition is given by $(B_0,R_0)\in \NN^{2d}$ with $B_0^{*}, R_0^{*}\geq 1$. Let $\{W(n)\}_{n\geq 1}$ be a positive sequence and let $W(0)\geq 0, p \in [0,1]$ be two constants. For any $n\in \NN$, conditional on $\mathcal{G}_n:=\sigma(B_m, R_m: m\leq n)$, define independent Bernoulli random variables $(\xi_{n+1}(i))_{1\leq i \leq d}$ by 
\begin{equation}
    \label{defxi}
    \PP(\xi_{n+1}(i)=1|\mathcal{G}_n)=\frac{p W(B_n^*)}{W(B_n^*)+W(R_n^*)}+\frac{(1-p)W(B_n(i))}{W(B_n(i))+W(R_n(i))}.
\end{equation}
Now set 
\begin{equation}
    \label{defBR}
    B_{n+1}(i):=B_{n}(i)+\xi_{n+1}(i), \quad R_{n+1}(i):=R_{n}(i)+1-\xi_{n+1}(i).
\end{equation}
The process $(B_n, R_n)_{n\in \NN}$ is called the interacting urn mechanism (IUM) with reinforcement sequence $\{W(n)\}_{n\in \NN}$ and interaction parameter $p$. We denote its law by $\PP_p^W$. 

\begin{figure}
    \centering
    \subfigure[barrier is kept with prob. $1-p$]{
        \begin{tikzpicture}[scale=1]
     
         \filldraw[line width=3pt, brown, fill=white] (0,1.8) -- (0,0) -- (2.5,0) -- (2.5,1.8);
         \filldraw[line width=3pt, brown, fill=white] (2.5,0) -- (5,0) -- (5,1.8);
     \node[circle, fill=black,text=white,font={\bfseries}](r1) at (0.9,1.1) {B};
     \node[circle, fill=black,text=white,font={\bfseries}](r2) at (2.1,0.5) {B};
     \node[circle, fill=red,text=white,font={\bfseries}](l1) at (0.5,0.5) {R};
     \node[circle, fill=black,text=white,font={\bfseries}](u1) at (1.3,0.5) {B};
     \node[circle, fill=black,text=white,font={\bfseries}](d1) at (1.7,1.1) {B};
     
     \node[above] at (1.3,1.9) {Urn 1};
     \node[above] at (3.7,1.9) {Urn 2};
     
     \node[circle, fill=red,text=white,font={\bfseries}](r3) at (3.4,1.1) {R};
     \node[circle, fill=black,text=white,font={\bfseries}](r4) at (4.6,0.5) {B};
     \node[circle, fill=red,text=white,font={\bfseries}](l2) at (3,0.5) {R};
     \node[circle, fill=black,text=white,font={\bfseries}](u2) at (3.8,0.5) {B};
     \node[circle, fill=red,text=white,font={\bfseries}](d2) at (4.2,1.1) {R}; 
     
     \end{tikzpicture}
 
         \label{fig2urn1p}
         }
\hspace{5pt}	
\subfigure[barrier is removed with prob. $p$]{
   \begin{tikzpicture}[scale=1]

    \filldraw[line width=3pt, brown, fill=white] (0,1.8) -- (0,0) -- (2.5,0);
    \filldraw[line width=3pt, brown, fill=white] (2.5,0) -- (5,0) -- (5,1.8);
\node[circle, fill=black,text=white,font={\bfseries}](r1) at (0.9,1.1) {B};
\node[circle, fill=black,text=white,font={\bfseries}](r2) at (2.1,0.5) {B};
\node[circle, fill=red,text=white,font={\bfseries}](l1) at (0.5,0.5) {R};
\node[circle, fill=black,text=white,font={\bfseries}](u1) at (1.3,0.5) {B};
\node[circle, fill=black,text=white,font={\bfseries}](d1) at (1.7,1.1) {B};

\node[circle, fill=red,text=white,font={\bfseries}](r3) at (3.4,1.1) {R};
\node[circle, fill=black,text=white,font={\bfseries}](r4) at (4.6,0.5) {B};
\node[circle, fill=red,text=white,font={\bfseries}](l2) at (3,0.5) {R};
\node[circle, fill=black,text=white,font={\bfseries}](u2) at (3.8,0.5) {B};
\node[circle, fill=red,text=white,font={\bfseries}](d2) at (4.2,1.1) {R}; 

\end{tikzpicture}

    \label{fig2urnp}
    }

\end{figure}

Unless otherwise specified, we assume $d=2$ (i.e. there are two urns) for simplicity. Without loss of generality, we assume that $R_0(1)\geq 1$ and $B_0(2)\geq 1$. Let
\begin{equation}
    \label{xinportiondef}
   x_n:=\frac{B_n(1)}{n+B_0(1)+R_0(1)}, \quad y_n:=\frac{B_n(2)}{n+B_0(2)+R_0(2)}
\end{equation}
be the proportions of black balls in the two urns at time $n$, respectively. 

For an IUM with $p>0$, there is a tendency for different components to adopt a common behavior. For example, for IUM with linear reinforcements, i.e. $W(n)=n$, Dai Pra, Louis and Minelli \cite{MR3217785} proved that if $B_0(1)=B_0(2)$ and $R_0(1)=R_0(2)$, then for any $p>0$, almost surely, 
\begin{equation}
    \label{commonlimit}
    \lim_{n\to \infty} x_n\ \text{and}\ \lim_{n\to \infty} y_n \quad \text{both exist and are equal.}
\end{equation}
This phenomenon is called synchronization. Moreover, it has been proved in \cite[Theorem 3.2]{MR3452818} that this common limit satisfies
\begin{equation}
    \label{nonatomcomlim}
    \PP(\lim_{n\to \infty} x_n\in \{0,1\})<1,\quad \text{and} \quad \PP(\lim_{n\to \infty} x_n=x)=0,\ \text{for any}\ x\in (0,1).
\end{equation}
However, IUMs with strong reinforcement may exhibit very different behaviors. We say that $\{W(n)\}_{n\in \NN}$ is a strong reinforcement sequence if
    \begin{equation}
        \label{strongreinforce}
        \sum_{n=1}^{\infty}\frac{1}{W(n)}<\infty.
    \end{equation} 
    
As we will see later, for some IUMs with strong reinforcement sequences and weak interaction (i.e. $p$ is small), one color can maintain its advantage in a single urn while it is at a disadvantage globally, in which case the urns do not synchronize. Indeed, this models a common phenomenon in economic systems: Many companies perform well in their local markets but struggle to replicate that success globally due to weak interactions between regions. On the other hand, for strong interaction, the phenomena of domination and monopoly may occur, as already exhibited in the ordinal dependent P\'olya urn.

\begin{definition}[Domination and monopoly]
    For an IUM, we denote by
      $$
  \mathcal{D}=\left\{\lim _{n \rightarrow \infty} (x_n,y_n)=(0,0)\right\} \cup\left\{\lim_{n \rightarrow \infty} (x_n,y_n)=(1,1)\right\}
  $$
  the event that eventually the number of balls of one color is negligible with respect to the
  number of balls of the other color, and call this event domination. Further, we denote by
  $$
  \mathcal{M}=\left\{R^{*}_{n+1}=R^{*}_n \text { eventually for all } n\right\} \cup\left\{B^{*}_{n+1}=B^{*}_n \text { eventually for all } n\right\}
  $$
  the event that eventually only balls of one color are added to the urns, and call this event monopoly. Note that $\mathcal{M} \subset \mathcal{D}$.
  \end{definition}
  
  We will be interested in how $\PP^{W}_p(\mathcal{D})$ and $\PP^{W}_p(\mathcal{M})$ are affected by the parameter $p$ and the sequence $\{W(n)\}_{n\in \NN}$, especially in the case of power function/polynomial reinforcement sequences, i.e. $W(n)=n^{\alpha}$, $n\geq 1$, for some real number $\alpha \geq 1$, or $W(n)$ is of the form
  \begin{equation}
      \label{defWpolyalpha}
      W(n):=n^{\alpha}+c_1n^{\alpha-1}+\cdots+c_{\alpha}, \quad n\geq 1,
  \end{equation}
  where $\alpha$ is a positive integer. For $\alpha=1$, we can assume that $c_1=0$ since in this case the IUM is equal in law to an IUM with $W(n)=n$ and shifted initial condition $(B_0, R_0)+(c_1,c_1,c_1,c_1)$. In addition, as we will see later, our results do not depend on $c_1, c_{\alpha-1},\cdots,c_{\alpha}$. Thus, by a slight abuse of notation, in either of the two cases above, we let $\PP_p^{(\alpha)}$ denote the law of the IUM with reinforcement sequence $\{W(n)\}_{n\in \NN}$ and parameter $p$. We show that under $\PP_p^{(\alpha)}$, domination implies monopoly a.s., and thus, we will not distinguish the two events in the sequel.
  
  \begin{proposition}
      \label{weakstrongdomi}
  Assume that $\{W(n)\}_{n\in \NN}$ satisfies (\ref{strongreinforce}) and is eventually increasing, i.e., there exists $N_1\in \NN$ such that $W(n+1)-W(n)\geq 0$ for all $n\geq N_1$. If 
  \begin{equation}
      \label{remKcond}
      \lim_{K\to \infty}\limsup_{n\to \infty}\left(\sum_{i\geq Kn}\frac{1}{W(i)}\right)\left(\sum_{i\geq n}\frac{1}{W(i)}\right)^{-1}=0,
  \end{equation}
  then $\PP_p^{W}(\mathcal{D} \backslash \mathcal{M})=0$ for any $p\in [0,1]$. In particular, for any $\alpha > 1$ and $p\in [0,1]$, one has $\PP_p^{(\alpha)}(\mathcal{D} \backslash \mathcal{M})=0$. 
  \end{proposition}

  For $\alpha=1$, by classical results for P\'olya urn model and (\ref{nonatomcomlim}), one has 
  \begin{equation}
      \label{P01D}
      \PP_0^{(1)}(\mathcal{D})=0, \quad  \PP_p^{(1)}(\mathcal{D})<1\ \text{for any } p>0.
  \end{equation}
  We will show in Theorem \ref{xnynu} (i) that $\PP_p^{(1)}(\mathcal{D})=0$ for $p>0$. Thus, $\PP_p^{(\alpha)}(\mathcal{D} \backslash \mathcal{M})=0$ also holds for $\alpha=1$.
  
   For $\alpha>1$, Launay proved in \cite[Theorem 2.3]{launay2012urns} that $\PP^{(\alpha)}_1(\mathcal{M})=1$ and conjectured that $\PP_p^{(\alpha)}(\mathcal{M})=1$ for all $p>0$. This work aims to disprove the conjecture and generalize the results obtained in \cite{launay2012urns}. 

\subsection{Main results}

\subsubsection{Power function/Polynomial reinforcements}

For $\alpha>1$, define 
\begin{equation}
    \label{defpmcritical}
    p_{\alpha}:=\inf\{0\leq q \leq 1: \PP_p^{(\alpha)}(\mathcal{D})=1, \ \forall p\geq q\}.
\end{equation}
which we call the critical parameter. Our first main result shows that $p_{\alpha}>0$, which disproves Launay's conjecture. Moreover, for $\alpha=1$, we prove that (\ref{commonlimit}) holds for general initial conditions, and the domination occurs with probability 0.

\begin{theorem}
    \label{xnynu}
  For any $\alpha\geq 1$ and $p\in [0,1]$, under $\PP_p^{(\alpha)}$, the sequence $(x_n,y_n)_{n\in \NN}$ defined in (\ref{xinportiondef}) is convergent a.s.. Moreover, \\
(i) if $\alpha=1$, then for any $p>0$, one has $\PP_p^{(1)}(\mathcal{D})=0$ and  
$$\lim_{n\to \infty}x_n=\lim_{n\to \infty}y_n, \quad a.s.$$   
 (ii) if $\alpha >1$, then 
    $$
    \PP_p^{(\alpha)}\left(\lim_{n\to \infty}(x_n,y_n)=(u_{\alpha, p},1-u_{\alpha, p})\right)>0, \quad \text{for} \ p\leq \frac{1}{6\alpha}\min\left(\alpha-1,2\right),
    $$
 where $u_{\alpha, p}$ is the unique solution of the following equation on $[0,1/2):$ 
    $$
    u=\frac{(1-p)u^{\alpha}}{u^{\alpha}+(1-u)^{\alpha}}+\frac{p}{2}.
    $$
\end{theorem}
Note that $u_{\alpha, p}$ exists if $p<(\alpha-1)/\alpha$, see Lemma \ref{lemsoluualpha}.

As mentioned in \cite[Conclusion]{launay2011interacting}, polynomial reinforcements do not behave as exponential reinforcements where $W(n)=\rho^n, n\in \NN$, for some $\rho>1$. It has been proved in \cite{launay2011interacting} that, for any $\rho>1$, one has 
$$\PP_p^W(\mathcal{M})=1, \ \text{if } p\geq \frac{1}{2}; \quad \PP_p^W(\mathcal{D}^c)\geq \PP_p^W(\lim_{n\to \infty}(x_n,y_n) = (p,1))>0,\ \text{if } p< \frac{1}{2}.$$ 
Therefore, one can observe a phase transition at $p=1/2$. 
\begin{remark}
    As $\rho$ tends to $\infty$, the exponential reinforcement mechanism converges to the "generalized" reinforcement, which was introduced and studied by Launay and Limic in \cite{launay2012generalized}.
\end{remark}

Our second main result shows that power function/polynomial reinforcements are weaker than exponential reinforcements in the sense that $p_{\alpha}<1/2$. On the other hand, for large $\alpha$, the reinforcement becomes very strong and behaves "like" the exponential reinforcement: For $p<1/2$, if $\alpha$ is sufficiently large, then with positive probability, $\lim_{n\to \infty}(x_n,y_n)$ is close to $(p,1)$. In particular, $\lim_{\alpha \to \infty}p_{\alpha}=1/2$.

\begin{theorem}
    \label{pmlarger12}
   (i) For any $\alpha >1$, one has $p_{\alpha}<\min(1/2,\alpha^{-1}(\alpha-1))$. \\
   (ii) For $\alpha\geq 3$, if $p\leq 1/2-\alpha^{-1}\log \alpha$, then $\PP_p^{(\alpha)}(\mathcal{D})<1$. \\
   (iii) Fix $p<1/2$, for sufficiently large $\alpha$, there exists $s_{\alpha} \in [0,1/2)\times (1/2,1]$ such that  
   $$\PP_p^{(\alpha)}(\lim_{n\to \infty}(x_n,y_n) = s_{\alpha})>0\quad \text{and} \quad \lim_{\alpha\to \infty}s_{\alpha} = (p,1).$$ 
\end{theorem}

\begin{remark}
    \label{palpharem}
For $\alpha>1$, by Theorem \ref{xnynu} (ii), we can define
\begin{equation}
    \label{supPDcpositive}
    \tilde{p}_{\alpha}:=\sup\{0\leq q \leq 1: \PP_p^{(\alpha)}(\mathcal{D})<1, \ \forall p\leq q\}.
\end{equation}
Then $\tilde{p}_{\alpha}\leq p_{\alpha}$. We conjecture that $\tilde{p}_{\alpha}= p_{\alpha}$. Theorem \ref{xnynu} and Theorem \ref{pmlarger12} give some estimates on their convergence rate to $0$, resp. $1/2$, as $\alpha \to 1$, resp. as $\alpha\to \infty$:
    $$ \frac{\alpha-1}{6 \alpha}\leq \tilde{p}_{\alpha} \leq  p_{\alpha} < \frac{\alpha-1}{\alpha}\quad (\alpha \leq 3); \quad  \frac{1}{2} - \frac{\log \alpha}{\alpha} \leq \tilde{p}_{\alpha}  \leq p_{\alpha} <\frac{1}{2}\quad (\alpha \geq 3).$$
\end{remark}

 \subsubsection{Urns with simultaneous drawing}
\label{urnsimulp1}

For general reinforcement sequences, the case $p=1$ was studied in \cite{launay2012urns} whose main result is the following. Recall (\ref{strongreinforce}) for the definition of the strong reinforcement sequences.

\begin{theorem}[Launay, \cite{launay2012urns}]
    \label{stronincfix}
    Given $d \geq 2$ urns, if $\{W(n)\}_{n\in \NN}$ is a non-decreasing strong reinforcement sequence, then $\PP_1^W(\mathcal{M})=1$.
\end{theorem}

We see from (\ref{Davissumcond}) that for $d=1$, the monotonicity assumption is not needed. It was conjectured in \cite{launay2012urns} that this assumption is also redundant for the cases $d\geq 2$. In this work, we show a generalization of Theorem \ref{stronincfix} to a larger class of strong reinforcement sequences.

We will not limit ourselves to two-color urns. In this case, it is convenient to assume that we have only a single urn of $N_c$-color balls where $N_c\geq 2$. Let $N_n(i)$ be the number of balls of the $i$-th color in this urn at time $n$, and write $N_n:=(N_n(i))_{1\leq i\leq N_c}$. Without loss of generality, we assume that $(N_0(i))_{1\leq i\leq N_c}=(a_i)_{1\leq i\leq N_c}$ are positive integers. For any $n\in \NN$, conditional on $\mathcal{G}_n:=\sigma(N_m:m\leq n)$, the law of $N_{n+1}-N_n$ is given by the multinomial distribution 
$$M(d; (p_1,p_2,\cdots, p_{N_c})), \quad \text{where}\ p_i:= \frac{W(N_n(i))}{\sum_{j=1}^{N_c} W(N_n(j))}.$$ 
One can easily see that the process $(N_n)_{n\in \NN}$ is an IUM with reinforcement sequence $\{W(n)\}_{n\in \NN}$ and parameter $p=1$. The event monopoly is then given by
$$
\mathcal{M}=\left\{\exists i \in \{1,2,\cdots, N_c\},\ s.t.\ N_{n+1}(i)=N_{n}(i)+d \text { for all large }n\right\}.
$$

\begin{theorem}
    \label{wnnewprop}
    One has $\PP_1^W(\mathcal{M})=1$ if $\{W(n)\}_{n\in \NN}$ is a strong reinforcement sequence that satisfies one of the following conditions:  \\ 
    (i) There exists a positive constant $C$ such that for all $n\geq 1$,
     \begin{equation}
         \label{Wnseccond}
           W(n)\delta_n \leq C, \ \text{where}\ \delta_n:=\sum_{k=n}^{\infty}\left|\frac{1}{W(k)}-\frac{1}{W(k+1)}\right|.
     \end{equation} 
     (ii) For $n\geq 1$, let $A_n:=\sum_{i=n}^{\infty}1/W(i)^2$. One has
     \begin{equation}
         \label{Wn2Anlarge}
        \limsup_{n\to \infty}\frac{\delta_n}{ \sqrt{A_n}}<\frac{1}{64(d-1)}, \quad \text{and}\quad \lim_{n\to \infty} \frac{1}{W(n) \sqrt{A_n}} =0.
     \end{equation}
 \end{theorem}
 \begin{remark}
    (I) If $\{W(n)\}_{n\in \NN}$ is non-decreasing and satisfies (\ref{strongreinforce}), then (\ref{Wnseccond}) holds with $C=1$. Thus, Theorem \ref{wnnewprop} generalizes Theorem \ref{stronincfix}.  \\
    (II) In either case, we require that $\delta_n$, the total variation of $\{1/W(k)\}_{k\geq n}$, is relatively small. For example, if $W(n)=(2+(-1)^n)n^2$ or $W(n)=e^{(2+(-1)^n)n}$, then neither (\ref{Wnseccond}) nor (\ref{Wn2Anlarge}) is satisfied. It is worth mentioning that similar conditions and examples have appeared in the study of strongly edge-reinforced random walks, see \cite{MR2349575}.  \\
    (III) Each condition cannot be derived from the other. By (I), $\{e^n\}_{n\in \NN}$ satisfies (\ref{Wnseccond}) but does not satisfy (\ref{Wn2Anlarge}). On the other hand, one can check by the Cauchy-Schwarz inequality that if 
    $$\sum_{n=1}^{\infty}\left(\frac{W(n+1)-W(n)}{W(n)}\right)^2<\infty,$$  
    then (\ref{Wn2Anlarge}) is satisfied, see e.g. the proof of \cite[Corollary 4]{MR2349575}. In particular, if $W(n)=2n^2+(-1)^nn^{4/3}$, then (\ref{Wn2Anlarge}) is satisfied but (\ref{Wnseccond}) is not satisfied.
 \end{remark}

   \section{Introduction to the proofs and the techniques}
\label{sectechniques}

\subsection{Notation}

  We let $C(a_1,a_2,\cdots,a_k)$ denote a positive constant depending only on real variables $a_1, a_2, \ldots, a_k$ and let $C$ denote a universal positive constant, which usually means that $C(a_1,a_2,\cdots,a_k)$ and $C$ do not depend on $n$. 
  
  For a real-valued function $h$ and a $[0,\infty)$-valued function $g$, we write $h(x)=O(g(x))$ as $x\to \infty$, resp. $x\to 0$, if there exist positive constants $C$ and $x_0$ such that $|h(x)| \leq C g(x)$ for all $x \geq x_0$, resp. $|x|\leq x_0$. 

 We let $\NN_+:=\NN \cap (0,\infty)$. We let $\|\cdot\|$ denote the usual Euclidean norm. We write $X \sim \operatorname{Exp}(\lambda)$ if a random variable $X$ has an exponential distribution with rate $\lambda$.

\subsection{Stochastic approximation algorithms} 

Under $\PP_p^{(\alpha)}$, we show that $((x_n,y_n))_{n\in \NN}$ defined by (\ref{xinportiondef}) is generated by a stochastic approximation algorithm (Robbins-Monro algorithm), and is closely related to the following (deterministic) planar nonlinear system 
\begin{equation}
    \label{odedxFm}   
    \left(\frac{dx}{dt},\frac{dy}{dt}\right)=F^{(\alpha)}_{p}(x,y)
\end{equation}
where $F^{(\alpha)}_{p}=(F_{p,1}^{(\alpha)},F_{p,2}^{(\alpha)})$ is a vector function on $[0,1]^2$ defined by
\begin{equation}
    \label{Falphapdef}
    \left\{\begin{aligned} F_{p,1}^{(\alpha)}(x,y)&:=-x +\frac{(1-p)x^{\alpha}}{x^{\alpha}+(1-x)^{\alpha}}+\frac{p(x+y)^{\alpha}}{(x+y)^{\alpha}+(2-x-y)^{\alpha}},  \\ F_{p,2}^{(\alpha)}(x,y)&:=-y +\frac{(1-p)y^{\alpha}}{y^{\alpha}+(1-y)^{\alpha}}+\frac{p(x+y)^{\alpha}}{(x+y)^{\alpha}+(2-x-y)^{\alpha}}.
    \end{aligned} \right.
\end{equation}
For an introduction to stochastic approximation algorithms, see e.g. \cite{MR1767993, MR1082341, MR2442439, MR1485774}. 

\begin{proposition}
    \label{propstocappxy}
    For any $\alpha\geq 1$ and $p\in [0,1]$, under $\PP_p^{(\alpha)}$, $((x_n,y_n))_{n\in \NN}$ defined by (\ref{xinportiondef}) satisfies the following recursion:
    \begin{equation}
        \label{stoappxy}
           (x_{n+1},y_{n+1})-(x_n,y_n)=\frac{1}{n+1}(F^{(\alpha)}_{p}(x_n,y_n)+\varepsilon_{n+1}+r_{n+1}), \quad n\in \NN,
    \end{equation}
    where $(\varepsilon_{n+1})_{n\in \NN}$ and $(r_{n+1})_{n\in \NN}$ are adapted sequences such that for all $n\in \NN$,
    $$
\EE(\varepsilon_{n+1}\mid \GG_n)=(0,0), \quad \|\varepsilon_{n+1}\| \leq 2, \quad \|r_{n+1}\| \leq \frac{C}{n+1},
    $$
where $C$ is a positive constant.
\end{proposition}

For the system (\ref{odedxFm}) with initial condition $(x(0),y(0)) \in [0,1]^2$, the existence and uniqueness of the solution follow from the Lipschitz property of $F^{(\alpha)}_{p}$ and Picard's theorem. Note that the solution satisfies $(x(t),y(t)) \in [0,1]^2$ for all $t\geq 0$. 

We shall study the asymptotic behavior of the solution to the system (\ref{odedxFm}).
It turns out that (\ref{odedxFm}) is a gradient system. The proof of the following result is direct and is omitted here.

\begin{proposition}
    \label{stogradsys}
    For any $\alpha \geq 1$, define $L^{(\alpha)}_{p}: [0,1]^2 \to \RR$ by
    \begin{equation}
        \label{defLm}
       L^{(\alpha)}_{p}(x,y)=(1-p)\left(G^{(\alpha)}(x)+G^{(\alpha)}(y)\right)+2pG^{(\alpha)}\left(\frac{x+y}{2}\right)-\frac{x^2+y^2}{2} ,
    \end{equation}
    where 
    $$
    G^{(\alpha)}(t):=\int_0^t \frac{u^{\alpha}}{u^{\alpha}+(1-u)^{\alpha}}du, \quad t \in [0,1].
    $$   
Then, we have $\operatorname{grad}L^{(\alpha)}_{p}=F^{(\alpha)}_{p}$.
\end{proposition}
\begin{example}[$\alpha=1,2$]
    \label{examLalpha}
  One has $L^{(1)}_p(x,y)= -p(x-y)^2/4$ and 
      $$
      \begin{aligned}
          L^{(2)}_p(x,y)&=\frac{1-p}{4} \log (x^2 +(1-x)^2)+\frac{1-p}{4} \log (y^2 +(1-y)^2) \\
          &+\frac{p}{2}\log ((x+y)^2+(2-x-y)^2)-p\log 2-\frac{x^2}{2}+\frac{x}{2}-\frac{y^2}{2}+\frac{y}{2}.
      \end{aligned}
      $$
  \end{example}

  A point $(x,y) \in [0,1]^2$ is called an equilibrium of (\ref{odedxFm}) if $F^{(\alpha)}_{p}(x,y)=0$. Let $\Lambda^{(\alpha)}_{p}$ be the set of all equilibrium points. Observe that $\Lambda^{(1)}_p=\{(x,x): x\in [0,1]\}$ for any $p>0$. We prove that $\Lambda^{(\alpha)}_{p}$ is a finite set for $\alpha>1$.  The cases $\alpha=3,p=0.4$ and $\alpha=5,p=0.3$ are plotted in Figure \ref{mpfinite} where $\Lambda^{(\alpha)}_{p}$ is the set of intersection points of the two curves.
\begin{figure}[t]
    \centering
    \subfigure[$\alpha=3, p=0.4$]{
    \includegraphics[width=6.6cm]{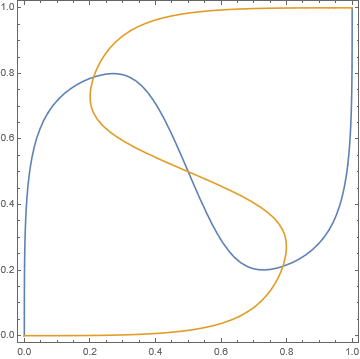}
    }
    \quad
    \subfigure[$\alpha=5, p=0.3$]{
    \includegraphics[width=6.6cm]{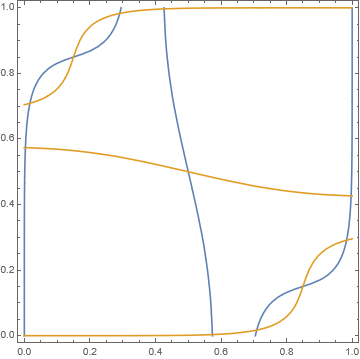}
    }
    \caption{$F^{(\alpha)}_{p,1}(x,y)=0$ in blue and $F^{(\alpha)}_{p,2}(x,y)=0$ in orange} \label{mpfinite}
\end{figure}

\begin{proposition}
    \label{equifinite}
   $\Lambda^{(\alpha)}_{p}$ includes $(0,0), (1/2,1/2),(1,1)$. Moreover, for $\alpha > 1$, one has,\\
   (i) if $p=0$, then $\Lambda^{(\alpha)}_{0} = \{0,1/2,1\} \times \{0,1/2,1\};$ \\
   (ii) if $p\in [1/2, (\alpha-1)/\alpha)$ (this is possible only when $\alpha>2$), then 
   $$
   \Lambda^{(\alpha)}_{p}=\{(0,0),(\frac{1}{2},\frac{1}{2}), (1,1),(u_{\alpha, p},1-u_{\alpha, p}),(1-u_{\alpha, p},u_{\alpha, p})\},
   $$
   where $u_{\alpha, p}$ is defined in Theorem \ref{xnynu}; \\
   (iii) if $p\geq (\alpha-1)/\alpha$, then $\Lambda^{(\alpha)}_{p}= \{(0,0), (1/2,1/2),(1,1)\};$ \\
     (iv) if $p>0$, then $\Lambda^{(\alpha)}_{p}$ is finite and
    $$\Lambda^{(\alpha)}_{p}\backslash \{(0,0), (\frac{1}{2},\frac{1}{2}),(1,1)\} \subset \left((0,\frac{1}{2})\times (\frac{1}{2},1)\right)  \cup \left((\frac{1}{2},1) \times (0,\frac{1}{2})\right).$$ 
\end{proposition}

 Propositions \ref{stogradsys}, \ref{equifinite} and Example \ref{examLalpha} then imply that the solution to (\ref{odedxFm}) converges to $\Lambda^{(\alpha)}_{p}$. More precisely,

    \begin{itemize}
        \item If $\alpha=1$ and $p>0$, then $x(t)+y(t)\equiv x(0)+y(0)$ and $L_p^{(1)}(x(t),y(t))$ increases to 0 as $t\to \infty$. In particular, both $x(t)$ and $y(t)$ converges to $(x(0)+y(0))/2$.
        \item If $\alpha>1$, $(x(t),y(t))$ converges to an equilibrium as $t\to \infty$.
    \end{itemize}

    \begin{definition}[Asymptotically stable equilibria]  
        For $\alpha>1$ and $p\in [0,1]$, an equilibrium $(x,y)$ of (\ref{odedxFm}) is said to be asymptotically stable if $(x,y)$ is a local maximum of $L^{(\alpha)}_{p}$ defined by (\ref{defLm}). We let $\mathcal{E}_{p}^{(\alpha)}$ be the set of asymptotically stable equilibria of (\ref{odedxFm}).
     \end{definition}
     
     Define 
     \begin{equation}
        \label{deffderxm}
       f(t):=\frac{t^{\alpha-1}(1-t)^{\alpha-1}}{[t^{\alpha}+(1-t)^{\alpha}]^2} , \quad t\in [0,1].
    \end{equation}
     Observe that the Jacobian matrix of the system (\ref{odedxFm}) is given by
     $$
     DF^{(\alpha)}_{p}(x,y)=\left(\begin{array}{ll}
         -1+\alpha(1-p)f(x)+\frac{\alpha p}{2}f(\frac{x+y}{2}) &\hspace{1.5cm} \frac{\alpha p}{2}f(\frac{x+y}{2}) \\
         \hspace{1.5cm} \frac{\alpha p}{2}f(\frac{x+y}{2}) & -1+\alpha(1-p)f(y)+\frac{\alpha p}{2}f(\frac{x+y}{2})
         \end{array}\right)
     $$
    which is a real symmetric matrix and thus has two real eigenvalues: 
     \begin{equation}
         \label{lambdapm}
         \begin{aligned}
           \lambda_{\pm}(x,y) &= -1 + \frac{\alpha p}{2}f(\frac{x+y}{2}) + \alpha(1-p)\frac{f(x)+f(y)}{2}\\
            &\pm \frac{1}{2}\sqrt{\alpha^2(1-p)^2(f(x)-f(y))^2+\alpha^2p^2 f^2(\frac{x+y}{2})}   .
         \end{aligned}
     \end{equation}
     Note that for an equilibrium $(x,y)$, if $ \lambda_{+}(x,y)<0$, then $(x,y)\in \mathcal{E}_{p}^{(\alpha)}$; it is called unstable if $\lambda_{+}(x,y) > 0$.
     
     \begin{example}
         \label{examunst}
 $(0,0)$ and $(1,1)$ are asymptotically stable equilibria since $
     \lambda_{+}(0,0)=\lambda_{+}(1,1)=-1$. While $(1/2,1/2)$ is unstable since $
     \lambda_{+}(1/2,1/2)=\alpha-1>0$.
     \end{example}

If $\alpha>1$ and $p$ is sufficiently small, or $p<1/2$ and $\alpha$ is sufficiently large, we prove the existence of asymptotically stable equilibrium.

     \begin{proposition}
        \label{stableexistence}
      (i) For $\alpha>1$ and $p\leq \alpha^{-1}\min((\alpha-1)/6,1/3)$, one has 
        $$
        \lambda_{+}(u_{\alpha, p},1-u_{\alpha, p}) <0,
        $$
        where $u_{\alpha, p}$ is defined in Theorem \ref{xnynu}. \\
        (ii) For $\alpha\geq 3$, if $0<p\leq  1/2-\alpha^{-1}\log \alpha$, then there exists an asymptotically stable equilibrium in $(0,1/2)\times (1/2,1)$. \\
    (iii) Fix $p\in (0,1/2)$, for sufficiently large $\alpha$, there exists $s_{\alpha} \in (0,1/2)\times (1/2,1)$ such that  
    $$\lambda_{+}(s_{\alpha})<0 \quad \text{and} \quad \lim_{\alpha\to \infty}s_{\alpha} = (p,1).$$ 
     \end{proposition}
   
     For the cases $\alpha=3,p=0.1$ and $\alpha=40,p=0.4$, the vector fields generated by (\ref{odedxFm}) are plotted in Figure \ref{psmallvec}. In Figure \ref{Fig3-1}, $(u_{\alpha, p},1-u_{\alpha, p})$ is inside the red circle. Readers can also find an asymptotically stable equilibrium $s_{\alpha}$ near $(0.4, 1)$ in Figure \ref{Fig40-4}.

     \begin{figure}[t]
        \centering
        \subfigure[$\alpha=3, p=0.1$]{
        \includegraphics[width=6.6cm]{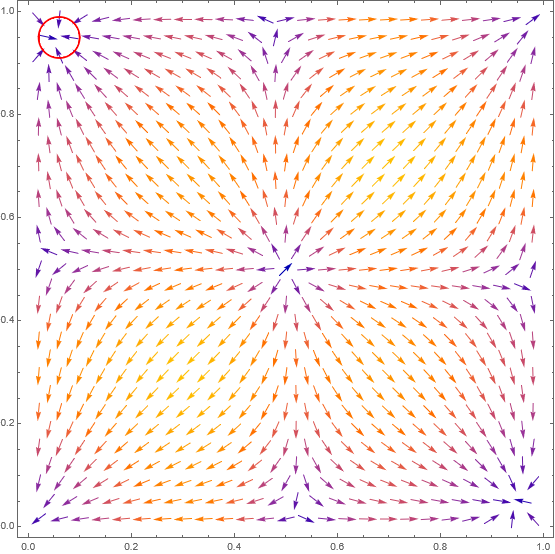}
       \label{Fig3-1}
        }
        \quad
        \subfigure[$\alpha=40, p=0.4$]{
            \includegraphics[width=6.6cm]{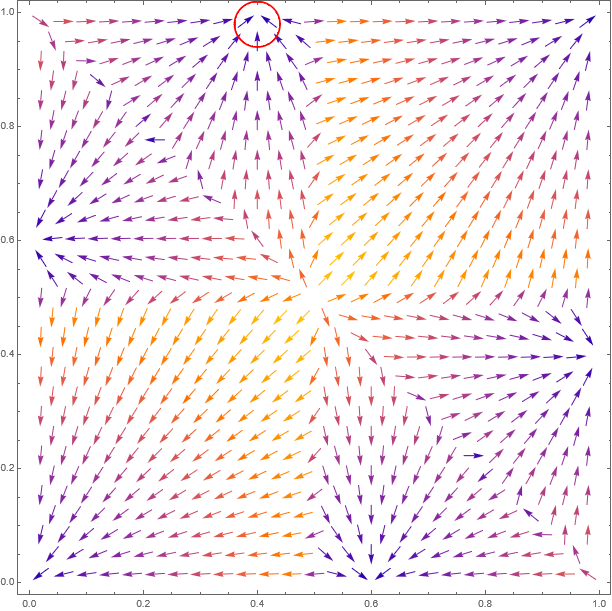}
            \label{Fig40-4}
            }
        \caption{The vector fields generated by (\ref{odedxFm}) for small p or large $\alpha$}\label{psmallvec}
    \end{figure} 

     The following result says that if $p\geq 1/2$, then all the equilibria except $(0,0)$ and $(1,1)$, are unstable, as is illustrated in Figure \ref{p12vecun} for the cases $\alpha=3,p=0.5$ and $\alpha=5,p=0.5$.

     \begin{corollary}
        \label{pgeq12unstalbe}
  Assume that $\alpha>1$ and $p\geq \min(1/2,\alpha^{-1}(\alpha-1))$. If $(x,y) \in \Lambda^{(\alpha)}_{p}\backslash\{(0,0),(1,1)\}$, then $\lambda_+(x,y)>0$.
\end{corollary}

     \begin{proof}
        If $p\in [1/2,\alpha^{-1}(\alpha-1))$ (in particular, $\alpha>2$), then 
        $$
        \lambda_+(u_{\alpha, p},1-u_{\alpha, p})=\lambda_+(1-u_{\alpha, p},u_{\alpha, p})=-1+\alpha p+\alpha(1-p)f(u_{\alpha, p})>\alpha p-1> 0
        $$
        For any $p\in [0,1]$, as shown in Example \ref{examunst}, $\lambda_+(1/2,1/2)>0$.
        Thus, Corollary \ref{pgeq12unstalbe} is a direct consequence of Proposition \ref{equifinite} (ii) and (iii).
        \end{proof}     

     \begin{figure}[t]
        \centering
        \subfigure[$\alpha=3, p=0.5$]{
        \includegraphics[width=6.6cm]{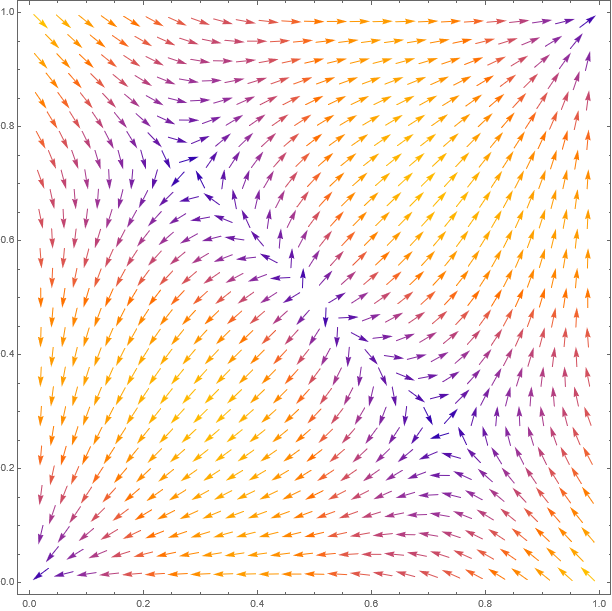}
        }
        \quad
        \subfigure[$\alpha=5, p=0.5$]{
        \includegraphics[width=6.6cm]{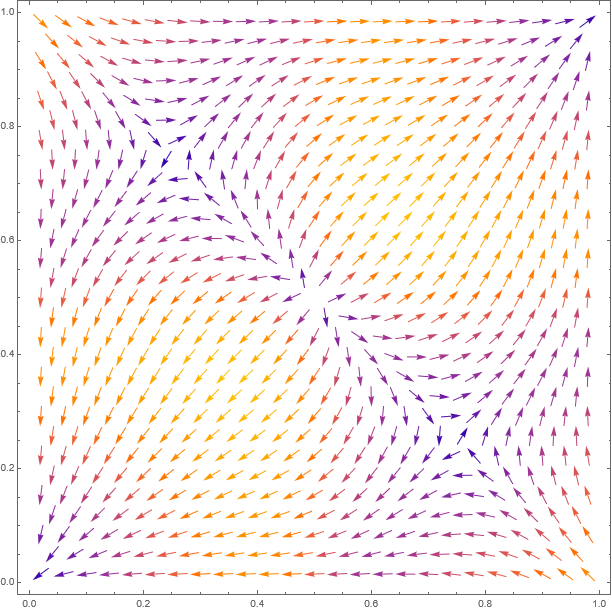}
        }
        \caption{The vector fields generated by (\ref{odedxFm}) for $p\geq 0.5$}\label{p12vecun}
    \end{figure}

     We now sketch the proof for Theorems \ref{xnynu} and \ref{pmlarger12} (the details will be given in Section \ref{stocappsec}): 
     \begin{itemize}
        \item Using stochastic approximation techniques, we show that under $\PP_p^{(\alpha)}$, as in the deterministic case (\ref{odedxFm}), almost surely, the sequence $((x_n,y_n))_{n\in \NN}$ is convergent and the limit belongs to $\Lambda^{(\alpha)}_{p}$. 
        \item  For $\alpha>1$, stochastic approximation theory can also be used to show that $(x_n,y_n)$ converges to any asymptotically stable equilibrium with positive probability, and converges to any unstable equilibrium with probability 0. Theorems \ref{xnynu} (ii) and Theorem \ref{pmlarger12} then follows from Propositions \ref{equifinite}, \ref{stableexistence} and Corollary \ref{pgeq12unstalbe}.
        \item For Theorems \ref{xnynu} (i),  we establish a finer stochastic approximation result for $(z_n)_{n\in \NN}$ in Lemma \ref{lemzn} where $z_n$ is the proportion of black balls in the whole system at time $n$:
        \begin{equation}
            \label{zndefpropB}
            z_n:=\frac{B_n^*}{2n+B_0^*+R_0^*}, \quad n\in \NN.
        \end{equation}
       This would enable us to show that 
            $$\PP_p^{(1)}(\lim_{n\to \infty} z_n \in \{0,1\})=0$$
     \end{itemize}
     
    \subsection{Continuous-time construction with time delays}
\label{constcondelay}

We use a continuous-time embedding technique to prove Theorem \ref{wnnewprop}. As we have mentioned, the case $d=1$ was solved by a continuous-time construction. It is natural to consider whether this technique can be generalized.

For the purpose of the proof, we introduce a new time-lines representation which we call continuous-time construction with time delays.

\begin{equation}
    \label{GNcloops}
    \begin{tikzpicture}[baseline=(current bounding box.center)]
\node   (a) [vertex] {x};
\path[scale=4] 
        (a) edge [loop right, "$e_1$ \ClockFrametrue\clock{1}{30}"]
        (a) edge [out=-30, in=-60, distance=5mm, ->, "$e_2$ \ClockFrametrue\clock{1}{20}"] (a)
        (a) edge [loop above, "$e_{N_c-1}$ \ClockFrametrue\clock{2}{35}"]
        (a) edge [out= 60, in= 30, distance=5mm, ->, "$e_{N_c}$ \ClockFrametrue\clock{3}{30}"] (a);
\draw [densely dotted]  (285:1.5em) arc (285:120:1.5em);
    \end{tikzpicture}
\end{equation}

We let $G$ be a graph as in (\ref{GNcloops}) consisting of a single vertex $x$ and $N_C \geq 2$ self-loops, i.e. the edge set $E=\{e_1,e_2.\cdots, e_{N_c}\}$ and $x \stackrel{e_i}{\sim} x$, $i=1,2,\cdots,N_c$. We shall define a continuous-time jump process $X=(X_t)_{t\geq 0}$ on $G$. Let us first introduce some preliminary notation. 

Let $0=\tau_0<\tau_1<\tau_2<\cdots$ be the hitting times of $X$ to $x$. For each $i \in \{1,2,\cdots,N_c\}$ and $n\in \NN$, let
\begin{equation}
    \label{defZni}
  Z_n(i):=\sum_{k=1}^n \mathds{1}_{\left\{(X_{\tau_{k-1}},X_{\tau_k})=e_i\right\}}+a_i
\end{equation}
be the number of visits to $e_i$ up to time $\tau_n$ plus $a_i \in \NN_+$ with the convention that $Z_0(i)=a_i$. Here $Z_{dn}(i)$ and $a_i$ should be interpreted respectively as the number of balls of the $i$-th color in the urn at time $n$ and the initial number of balls of the $i$-th color (see Proposition \ref{ZkdNin} for a more precise statement). For $n\geq a_i$, let
\begin{equation}
    \label{nvisittimeei}
    \sigma_n(i):=\inf\{\tau_{m}: Z_m(i)\geq n\}
\end{equation}
with the convention that $\inf \emptyset=\infty$. Let $\{\xi^{(i)}_n\}_{1\leq i\leq N_c, n>a_i}$ be independent Exp(1)-distributed random variables. The law of $X$ is defined as follows:

At time $t=0$, on each edge $e_i$ $(1\leq i \leq N_c)$, we launch a timer with a duration $\xi^{(i)}_{a_i+1}/W(a_i)$. When the timer of an edge $e_i$ rings, $X$ jumps to cross $e_i$ instantaneously. If an edge $e_i$ is crossed at time $\tau_n$ such that $kd < n \leq (k+1)d$ for some $k\in \NN$, then we launch a new timer on this edge with a duration $\xi^{(i)}_{Z_n(i)+1}/W(Z_{kd}(i))$. For $k\geq 1$, at time $\tau_{kd}$, we update the denominators (i.e. the rates) for all the timers: for $i \in \{1,2,\cdots,N_c\}$, if the timer on $e_i$ has run a time of $t^{(i)}_k$, then we reset the timer such that the remaining time becomes 
\begin{equation}
    \label{updatedeno}
    \frac{\xi^{(i)}_{Z_{kd}(i)+1}-W(Z_{(k-1)d}(i))t^{(i)}_k}{W(Z_{kd}(i))}.
\end{equation}

\begin{remark}
    \label{updateremark}
    (i) Just before we reset the timers, the remaining time of the timer on $e_i$ is 
    $$\frac{\xi^{(i)}_{Z_{kd}(i)+1}}{W(Z_{(k-1)d}(i))}-t^{(i)}_k.$$
    (ii) If, for some $j \in \{1,2,\cdots,N_c\}$, $e_j$ is not crossed during the time interval $(\tau_{(k-1)d},\tau_{kd}]$, then $Z_{(k-1)d}(j)=Z_{kd}(j)$ so that there is nothing to change for the timer on $e_j$. \\
 (iii) If $X$ jumps to cross $e_j$ for some $j \in \{1,2,\cdots,N_c\}$ at time $\tau_{kd}$, we will launch a new timer on $e_j$ and thus $t^{(j)}_k=0$. We may simply launch a new timer on $e_j$ with a duration $\xi^{(j)}_{Z_{kd}(j)+1}/W(Z_{kd}(j))$ rather than $\xi^{(j)}_{Z_{kd}(j)+1}/W(Z_{(k-1)d}(j))$.\\
    (iv) The timer which corresponds to $\xi^{(i)}_{n+1}$ may run at different rates as time changes. All the possible denominators (rates) are 
    $$
W(n-\ell), \quad \ell \in \{0,1,\cdots,d-1\} \ \text{such that}\ \ell \leq n-a_i,
    $$
    due to the time delays. Note that we may update this timer at jumping times but we will never launch a new one until it rings. Recall $\sigma_n(i)$ defined in (\ref{nvisittimeei}). If $\sigma_{n+1}(i)<\infty$, the total time this timer needs to run is simply $\sigma_{n+1}(i)-\sigma_{n}(i)$, in which case we can write
    \begin{equation}
        \label{decomsigman}
         \sigma_{n+1}(i)-\sigma_{n}(i)=\sum_{\ell=0}^{(d-1)\wedge (n-a_i)}\frac{b_{\ell}}{W(n-\ell)},
    \end{equation}
    where $b_{\ell}\geq 0$ and $\sum_{\ell=0}^{(d-1)\wedge (n-a_i)} b_{\ell}=\xi^{(i)}_{n+1}$.
\end{remark}

We denote the natural filtration of $X$ by $(\FF_t)_{t\geq 0}$, i.e. $\FF_t:=\sigma(X_s: 0\leq s\leq t)$. Recall the process $(N_k(i))_{1\leq i \leq N_c, k\in \NN}$ defined in Section \ref{urnsimulp1}.
\begin{proposition}
    \label{ZkdNin}
    Let $X$ be the jump process defined above. Then, 
    $$
    (Z_{kd}(i))_{1\leq i \leq N_c, k\in \NN} \stackrel{\mathcal{L}}{=} (N_k(i))_{1\leq i \leq N_c, k\in \NN}
    $$
    In particular, we may define $X$ and the IUM on the same probability space such that a.s.
\begin{equation}
    \label{ZkdNki}
    (Z_{kd}(i))_{1\leq i \leq N_c, k\in \NN} = (N_k(i))_{1\leq i \leq N_c, k\in \NN}
\end{equation}
\end{proposition}
\begin{proof}
   Note that conditional on $\FF_{\tau_{kd}}$, by the memoryless property of exponentials, (\ref{updatedeno}) has the same distribution as $\xi/W(Z_{kd}(i))$ where $\xi$ is an Exp(1)-distributed random variable. The rest of the proof is similar to that of \cite[Lemma 3.1]{tarres2011localization}.
\end{proof}

As we explained in Remark \ref{updateremark}, unlike the continuous-time construction in the proof of \cite[Theorem 3.6]{MR2282181}, at time $\tau_n \in (\tau_{kd},\tau_{(k+1)d})$, we keep using the data we collect at time $\tau_{kd}$ (i.e. $(Z_{kd}(i))_{1\leq i \leq N_c}$) to launch new timers. This justifies its name continuous-time construction with time delays. It is a powerful technique that allows us to give a very short proof of a multi-color ($N_c\geq 2$) version of Theorem \ref{stronincfix}. 

\begin{proof}[A new proof of Theorem \ref{stronincfix} with $N_c\geq 2$]
    Conditional on $\FF_{nd}$, if $Z_{nd}(i)= \max_{1\leq j\leq N_c}\{Z_{nd}(j)\}$ for some $i$, then the probability that $Z_{(n+1)d}(i)=Z_{nd}(i)+d$ (i.e. we add $d$ balls of the relative major color at time $n+1$) is lower bounded by $(1/N_c)^d$ since $\{W(n)\}_{n\in \NN}$ is non-decreasing. By the conditional Borel-Cantelli lemma, see e.g. \cite{MR0837655}, a.s. such an event occurs for infinitely many $n$, and thus, there is an infinite sequence of finite stopping times $\{\tau_{n_kd}\}_{k\geq 1}$ such that at time $\tau_{n_kd}$, there exists $i_k \in \{1,2,\cdots,N_c\}$ such that
    \begin{equation}
        \label{Znkddad}
        Z_{n_kd}(i_k) \geq d + \max_{j\neq i_k}\{Z_{n_kd}(j)\}.
    \end{equation}
    By (\ref{decomsigman}), for $j=1,2,\cdots,N_c$, if $\sigma_n(j)<\infty$,
    \begin{equation}
        \label{ineholdingtimeincrease}
        \frac{\xi_{n}^{(j)}}{W(n-1)} \leq \sigma_n(j)-\sigma_{n-1}(j) \leq \frac{\xi_{n}^{(j)}}{W(n-d)}, \quad \forall n\geq a_j+d.
    \end{equation}
   As is mentioned in the proof of Proposition \ref{ZkdNin}, conditional on $\FF_{\tau_{n_kd}}$, the remaining time of the timer on $e_j$ has the distribution of an independent copy of $\xi/W(Z_{n_kd}(j))$ where $\xi$ is an Exp(1)-distributed random variable. By a slight abuse of notation, the time remaining is denoted by $\xi^{(j)}_{Z_{n_kd}(j)+1}/W(Z_{n_kd}(j))$. 
    By symmetry and (\ref{Znkddad}), conditional on $\FF_{\tau_{n_kd}}$, with probability at least $1/N_c$,
    $$
    \sum_{n=Z_{n_kd}(i_k)+1}^{\infty}\frac{\xi_{n}^{(i_k)}}{W(n-d)}  < \min_{j\neq i_k}\left\{ \sum_{n=Z_{n_kd}(j)+1}^{\infty}\frac{\xi_{n}^{(j)}}{W(n-1)}\right\}
    $$
    (note that all the sums above all have continuous distributions) and in particular, by (\ref{ineholdingtimeincrease}),
    $$
\lim_{n\to \infty}\sigma_n(i_k)-\tau_{n_kd}<  \min_{j\neq i_k}\{ \lim_{n\to \infty}\sigma_n(j)-\tau_{n_kd} \}.
    $$
    That is, the remaining time needed to visit $e_{i_k}$ i.o. is strictly less than that needed for any other edge. By (\ref{ZkdNki}) in Proposition \ref{ZkdNin}, this is equivalent to saying that only balls of color $i_k$ are taken infinitely often (after time $n_k$). Therefore, for any $k\geq 1$,
       $$
    \PP_1^{W}(\mathcal{M}\mid \FF_{\tau_{n_kd}}) \geq  \frac{1}{N_c}, \quad \forall k\geq 1.
       $$
    We then conclude that $\PP_1^W(\mathcal{M})=1$ by Levy's 0-1 law.
\end{proof}

We will show in Section \ref{ctdelays} that one can apply a similar argument if 
$\delta_n$ is small in the sense of (\ref{Wnseccond}) or (\ref{Wn2Anlarge}), which enables us to prove Theorem \ref{wnnewprop}.

\subsection{Coupling} 

Proposition \ref{weakstrongdomi} is proved by coupling. Let $(B_n,R_n)_{n\in \NN}$ be an IUM with reinforcement sequence $\{W(n)\}_{n\in \NN}$ and interaction parameter $p$. We define a new urn process $(\tilde{B}_n,\tilde{R}_n)_{n\in \NN}$ as follows, where $\tilde{B}_n(i)$ and $\tilde{R}_n(i)$ ($i=1,2$) denote the number of black and red balls in the $i$-th urn at time $n$, respectively. 

Similarly, we write $\tilde{B}_n:=(\tilde{B}_n(1),\tilde{B}_n(2))$, $\tilde{R}_n:=(\tilde{R}_n(1),\tilde{R}_n(2))$ and $\tilde{B}_n^{*}:=\tilde{B}_n(1)+\tilde{B}_n(2)$, $\tilde{R}_n^{*}:=\tilde{R}_n(1)+\tilde{R}_n(2)$. The initial composition is given by $(\tilde{B}_0,\tilde{R}_0) \in \NN_+^{4}$. For any $n\in \NN$, at time step $2n+1$, we add a black ball to the first urn with probability 
$$
\frac{W(\tilde{B}_{2n}(1))}{W(\tilde{B}_{2n}(1))+W(\tilde{R}_{2n}^{*})},
$$
otherwise, we add a red ball to the first urn; at time step $2n+2$, we add a black ball to the second urn with probability 
$$
\frac{W(\tilde{B}_{2n}(2))}{W(\tilde{B}_{2n}(2))+W(\tilde{R}_{2n+1}^{*})},
$$
otherwise, we add a red ball to the second urn.

    In words, red balls are always drawn from all the urns combined, black balls are always drawn from the urn alone. Compared to the IUM, it is natural to expect that there will be more red balls and fewer black balls if $\{W(n)\}_{n\in \NN}$ is non-decreasing.

    \begin{lemma}
        \label{couplcomparenewurn}
        Assume that $\{W(n)\}_{n\in \NN}$ is non-decreasing and $(\tilde{B}_0,\tilde{R}_0)=(B_0,R_0)$. Then we can define the two urn processes $(B_n,R_n)_{n\in \NN}$ and $(\tilde{B}_n,\tilde{R}_n)_{n\in \NN}$ above on the same probability space such that 
        \begin{equation}
            \label{coupleine}
             \tilde{R}_{2n}(i) \geq R_n(i),\ \tilde{B}_{2n}(i) \leq B_n(i), \quad \forall n\in \NN, i=1,2.
        \end{equation}
    \end{lemma}

    \begin{lemma}
        \label{estnewurntimelines}
    Assume that  $\{W(n)\}_{n\in \NN}$ is non-decreasing and satisfies (\ref{strongreinforce}) and (\ref{remKcond}). Then, there exists positive constants $\varepsilon_1 \in (0,1)$ and $\kappa \in \NN_+$ such that if $n\geq \kappa$ and $\tilde{R}_{2n}^{*}<\varepsilon_1n$, then 
    $$
    \PP(  \lim_{\ell \to \infty}\tilde{R}^*_{\ell}<\infty \mid \tilde{\FF}_{2n})>\frac{1}{10e},
    $$
    where $\tilde{\FF}_n:=\sigma(\tilde{B}_m(i), \tilde{R}_m(i): 1\leq i \leq 2, m\leq n)$.
    \end{lemma}

    Lemmas \ref{couplcomparenewurn} and \ref{estnewurntimelines} will be proved in Section \ref{seccouple}. Notice that they also hold if we interchange the colors black/red. Using Lemmas \ref{couplcomparenewurn} and \ref{estnewurntimelines}, we can prove Proposition \ref{weakstrongdomi}.

    \begin{proof}[Proof of Proposition \ref{weakstrongdomi}]
        Fix $\varepsilon\in (0,1)$, let $\varepsilon_1,\kappa$ be as in Lemma \ref{estnewurntimelines} for $\{W(n)\}_{n\geq N_1}$. We define an infinite sequence of stopping times $\{T_n\}_{n\in \NN}$ as follows. Let $T_0:=\kappa$ and for any $n \in \NN$,
        $$
     T_{n+1}:=\inf\{m>T_n: N_1\leq \min(B^{*}_{2m},R^{*}_{2m})<\varepsilon_1 m\}
        $$
        with the convention that $\inf \emptyset =\infty$. If $T_n=\infty$ for some $n$, then we set $T_{k}:=\infty$ for all $k>n$. By Lemma \ref{couplcomparenewurn} and Lemma \ref{estnewurntimelines}, for any $n\geq 1$,
       $$
    \PP_p^{W}(\mathcal{D}^c\cup\mathcal{M}\mid \GG_{T_n})\mathds{1}_{\{T_n <\infty\}} \geq \PP_p^{W}(\mathcal{M}\mid \GG_{T_n })\mathds{1}_{\{T_n<\infty\}} \geq \frac{1}{10e}\mathds{1}_{\{T_n<\infty\}}.
       $$
       On the other hand, on the event $\{T_n=\infty\}$, either 
       \begin{enumerate}
        \item $\min(B^{*}_{2m},R^{*}_{2m})<N_1$ for all $m$, in which case the monopoly occurs, or
        \item $\min(B^{*}_{2m},R^{*}_{2m})\geq \varepsilon_1 m$ for all large $m$, in which case domination does not occur.
       \end{enumerate}
         Therefore, $\PP_p^{W}(\mathcal{D}^c\cup\mathcal{M}\mid \GG_{T_n})\mathds{1}_{\{T_n=\infty\}}= \mathds{1}_{\{T_n=\infty\}}$.
       Thus, for any $n \geq 1$,
       $$
       \PP_p^{W}(\mathcal{D}^c\cup\mathcal{M}\mid \GG_{T_n})>\frac{1}{10e}.
       $$
    By Levy's 0-1 law, $\PP_p^{W}(\mathcal{D}^c\cup\mathcal{M})=1$ which completes the proof since $\mathcal{M} \subset \mathcal{D}$.
    \end{proof}

\subsection{Organization of the remaining of this paper} 

Section \ref{oderesults} concerns the results on the deterministic nonlinear system (\ref{odedxFm}): Proposition \ref{equifinite} and Proposition \ref{stableexistence} are proved.

Section \ref{stocappsec} develops the framework relating the behavior of $(x_n,y_n)$ to the planar system (\ref{odedxFm}): We prove Proposition \ref{propstocappxy}, Theorem \ref{xnynu} and Theorem \ref{pmlarger12}. 

Theorem \ref{wnnewprop} is proved in Section \ref{ctdelays}. Section \ref{seccouple} is devoted to the proofs of Lemma \ref{couplcomparenewurn} and Lemma \ref{estnewurntimelines}. Some open problems are presented in the last section. 

\section{Results on the deterministic dynamical system}
\label{oderesults}

We assume that $\alpha>1$ and $p\in [0,1]$. The following two functions will be used frequently:
 \begin{equation}
    \label{defghalpha}
        h(t):=\frac{t^{\alpha}}{t^{\alpha}+(1-t)^{\alpha}}, \quad g(t):=-t+(1-p)h(t)+\frac{p}{2}, \quad t\in [0,1].
\end{equation}
In particular, $h^{\prime}(t)=\alpha f(t)$ and $g^{\prime}(t)=-1+\alpha (1-p) f(t)$ where $f(t)$ is defined by (\ref{deffderxm}).
\begin{lemma}
    \label{lemsoluualpha}
    The equation $g(u)=0$ has a solution on $[0,1/2)$ if and only if $p <  (\alpha-1)/\alpha$. If one solution exists, then it is unique.
\end{lemma}
\begin{proof}
 The assertion is trivial for $p=1$. We now assume that $p<1$. Note that $g(0)=p/2, g(1/2)=0$. Observe that $g^{\prime}(t)$ is a strictly increasing function on $[0,1/2]$ with $g^{\prime}(0)=1$ and $g^{\prime}(1/2)=\alpha(1-p)-1$. 

 If $g^{\prime}(1/2)>0$, then there exists a unique $t_0 \in (0,1/2)$ such that $g^{\prime}(t_0)=0$. The function $g$ is strictly decreasing on $[0,t_0]$ and strictly increasing on $[t_0,1/2]$. In particular, there exists a unique solution to $g(u)=0$ on $[0,1/2)$. See Figure \ref{Fig225} for the case $\alpha=2$ and $p=0.25$.

 If $g^{\prime}(1/2)\leq 0$, then $g$ is strictly decreasing on $[0,1/2]$, and thus, there is no solution to $g(u)=0$ on $[0,1/2)$. The case $\alpha=2$ and $p=0.75$ is plotted in Figure \ref{Fig275}.
\end{proof}

\begin{figure}[t]
    \centering
    \subfigure[$\alpha=2, p=0.25$]{
    \includegraphics[width=6.6cm]{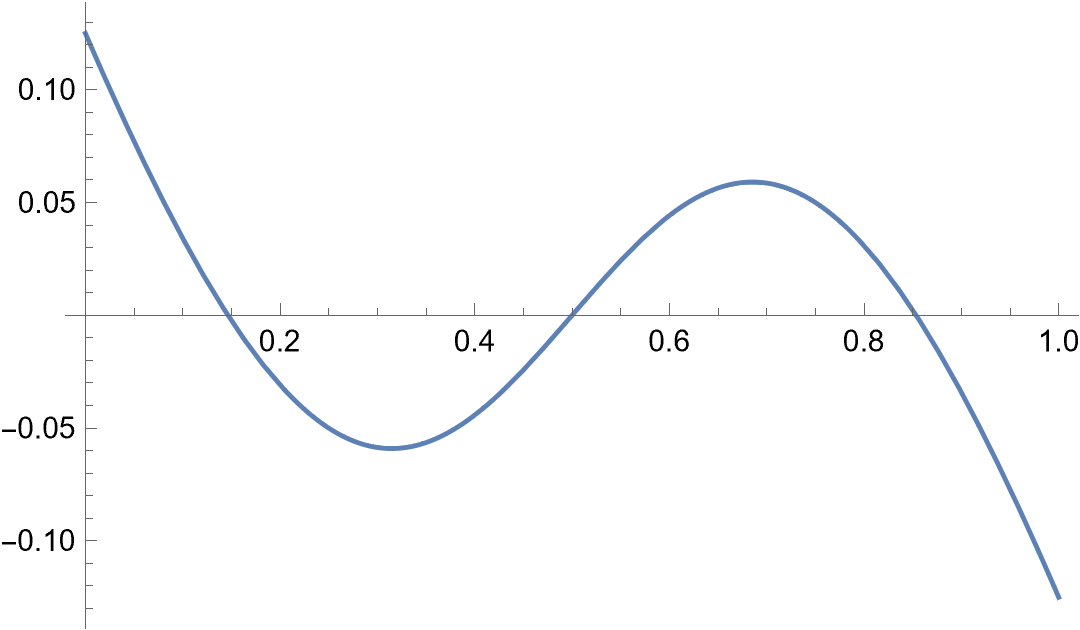}
    \label{Fig225}
    }
    \quad
    \subfigure[$\alpha=2, p=0.75$]{
    \includegraphics[width=6.6cm]{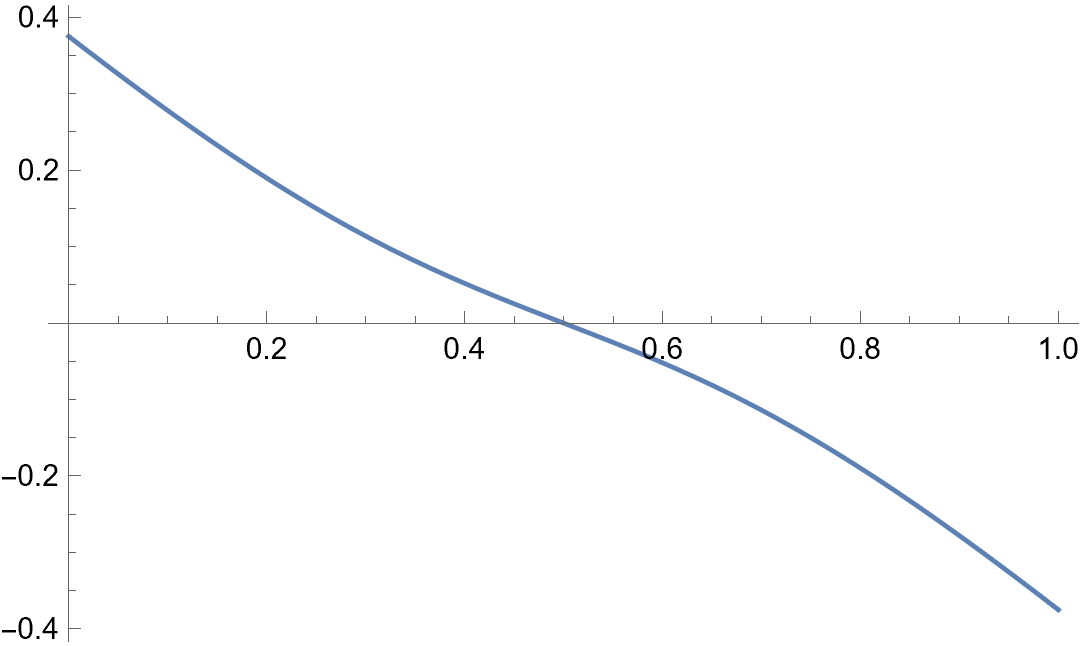}
    \label{Fig275}
    }
    \caption{$g(u)=-u+\frac{(1-p)u^{\alpha}}{u^{\alpha}+(1-u)^{\alpha}}+\frac{p}{2}$ on $[0,1]$}\label{guualphap}
\end{figure}

We denote the unique solution by $u_{\alpha, p}$ when it exists. Note that $u_{\alpha, p}=0$ only when $p=0$. 
 
The proof of Proposition \ref{equifinite} will need the following three technical lemmas.

\begin{lemma}
   \label{simplecaselem}
(i) If $\alpha>1$ and $p=0$, then $\Lambda^{(\alpha)}_{0} = \{0,1/2,1\} \times \{0,1/2,1\}$. \\
(ii) For $\alpha> 1$ and $p>0$, one has
$$\Lambda^{(\alpha)}_{p} \bigcap \left( [0,\frac{1}{2}]^2\cup [\frac{1}{2},1]^2 \cup \partial [0,1]^2\right)=\{(0,0), (\frac{1}{2},\frac{1}{2}),(1,1)\}.$$
\end{lemma}
\begin{proof}
 (i) If $z\in [1/2,1]$, then 
   \begin{equation}
       \label{zmlargez}
      \frac{z^{\alpha}}{z^{\alpha}+(1-z)^{\alpha}}= \frac{1}{1+(\frac{1}{z}-1)^{\alpha}} \geq  \frac{1}{1+(\frac{1}{z}-1)}=z ,
   \end{equation}
   where the equality holds if and only if $z=1/2$ or $z=1$. The inequality (\ref{zmlargez}) is reversed if $z\in [0,1/2]$ where the equality holds if and only if $z=0$ or $z=1/2$. This proves (i).

(ii) If $(x,y)\in [1/2,1]^2$ is an equilibrium, then by (\ref{zmlargez}), we have
$$0=-x+ \frac{(1-p)x^{\alpha}}{x^{\alpha}+(1-x)^{\alpha}} + \frac{p(\frac{x+y}{2})^{\alpha}}{(\frac{x+y}{2})^{\alpha}+(1-\frac{x+y}{2})^{\alpha}}\geq -x+(1-p)x+\frac{p(x+y)}{2}=\frac{p(y-x)}{2},$$
and similarly,
$$0=-y+ \frac{(1-p)y^{\alpha}}{y^{\alpha}+(1-y)^{\alpha}} + \frac{p(x+y)^{\alpha}}{(x+y)^{\alpha}+(2-x-y)^{\alpha}}\geq \frac{p(x-y)}{2}.$$
Thus, $(x,y)=(1/2,1/2)$ or $(1,1)$. Similarly, if $(x,y)\in [0,1/2]^2$ is an equilibrium, then $(x,y)=(0,0)$ or $(1/2,1/2)$. Moreover, it is easy to see that if a boundary point $(x,y)\in \partial [0,1]^2$ is an equilibrium, then $(x,y)=(0,0)$ or $(1,1)$. (ii) is then proved.
\end{proof}

\begin{lemma}
    \label{betahpositve}
    Define 
    \begin{equation}
        \label{equfhposi}
        \beta(t) := \frac{1}{2} \frac{(1+t)^{\alpha}-(1-t)^{\alpha}}{(1+t)^{\alpha}+(1-t)^{\alpha}}+\frac{1}{2}\frac{t^{\alpha}}{t^{\alpha}+(1-t)^{\alpha}}-t, \quad  t\in (0,\frac{1}{2}).
    \end{equation} 
If $\alpha\geq 2$, then for any $t\in (0,1/2)$, one has $\beta(t)>0$.
\end{lemma}
\begin{proof}
Observe that on the interval $(0,1/2)$, the function $h$ defined in (\ref{defghalpha}) is convex and $h(t)<t$. Thus, if $t\in (1/4,1/2)$, then
$$
\frac{h(t)+h(1-2t)}{2}-h(\frac{1-t}{2})\geq 0 >\frac{1}{2}(h(1-2t)-1+2t).
$$
In particular,
$$
\beta(t)=\frac{1}{2}h(t)-h(\frac{1-t}{2}) + \frac{1}{2}-t>0, \quad t\in (\frac{1}{4},\frac{1}{2}).
$$
Since $\alpha\geq 2$, one has
$$
(1+\frac{2t}{1-t})^{\alpha}=\left(1+\frac{4t}{1-t}+\frac{4t^2}{(1-t)^2}\right)^{\frac{\alpha}{2}}\geq  1+\frac{2\alpha t}{1-t}+\frac{2\alpha t^2}{(1-t)^2}, \quad t\in (0,\frac{1}{2}).
$$
Thus, for $\alpha\geq 9/4$ and $t\in (0,1/4]$, we have 
   $$ \frac{(1+t)^{\alpha}-(1-t)^{\alpha}}{(1+t)^{\alpha}+(1-t)^{\alpha}}=\frac{(1+\frac{2t}{1-t})^{\alpha}-1}{(1+\frac{2t}{1-t})^{\alpha}+1}\geq \frac{\alpha t}{(1-t)^2+\alpha t}\geq 2t,$$
   where in the last inequality we used that 
$$
\alpha\geq \frac{9}{4}\geq \frac{2}{1-(\frac{t}{1-t})^2}=\frac{2(1-t)^2}{1-2t}.
$$
If $\alpha\in [2,9/4)$ and $t\in (0,1/4]$, then 
$$
\begin{aligned}
    \beta(t)&\geq \frac{1}{2} \frac{(1+t)^{2}-(1-t)^{2}}{(1+t)^{2}+(1-t)^{2}}+\frac{1}{2}\frac{t^{\frac{9}{4}}}{t^{\frac{9}{4}}+(1-t)^{\frac{9}{4}}}-t \\
    &=\frac{t^{\frac{9}{4}}}{2}\left(\frac{-2t^{\frac{3}{4}}}{1+t^2}+\frac{1}{t^{\frac{9}{4}}+(1-t)^{\frac{9}{4}}} \right) \geq \frac{t^{\frac{9}{4}}}{2}\left(\frac{-\sqrt{2}}{2}+1 \right)>0,
\end{aligned}
$$
which completes the proof.
\end{proof}

\begin{lemma}
    \label{plarge12equili}
    Let $\alpha>2$ and $p\in [1/2, (\alpha-1)/\alpha)$. One has 
    $$
    \Lambda^{(\alpha)}_{p}=\{(0,0),(\frac{1}{2},\frac{1}{2}), (1,1),(u_{\alpha, p},1-u_{\alpha, p}),(1-u_{\alpha, p},u_{\alpha, p})\} ;
    $$
\end{lemma}
\begin{proof}
   By Lemma \ref{lemsoluualpha}, the set of equilibria on the line $x+y=1$ is given by
    $$
    \{(\frac{1}{2},\frac{1}{2}), (u_{\alpha, p},1-u_{\alpha, p}),(1-u_{\alpha, p},u_{\alpha, p})\} .
$$
 Thus, it remains to show that
   $$
   \Lambda^{(\alpha)}_{p}\backslash (\{(x,y):x+y=1\}\cup \{(0,0), (1,1)\}) =\emptyset.
   $$
   Observe that if $(x,y)\in \Lambda^{(\alpha)}_{p}$, then $(1-x,1-y)\in \Lambda^{(\alpha)}_{p}$. Therefore, by Lemma \ref{simplecaselem}, it suffices to prove that there is no equilibrium $(\tilde{x},\tilde{y})$ such that $0<\tilde{x}<1/2<\tilde{y}<1$ and $\tilde{z}:=(\tilde{x}+\tilde{y})/2\in (1/2,3/4)$.  We argue by contradiction and assume that $(\tilde{x},\tilde{y})$ is such an equilibrium.

   Since  $(\tilde{x},\tilde{y})\in \Lambda^{(\alpha)}_{p}$,
    \begin{equation}
        \label{zsumequ}
       -2\tilde{z} + \frac{(1-p)(2\tilde{z}-\tilde{y})^{\alpha}}{(2\tilde{z}-\tilde{y})^{\alpha}+(1+\tilde{y}-2\tilde{z})^{\alpha}}+\frac{(1-p)\tilde{y}^{\alpha}}{\tilde{y}^{\alpha}+(1-\tilde{y})^{\alpha}}+\frac{2p\tilde{z}^{\alpha}}{\tilde{z}^{\alpha}+(1-\tilde{z})^{\alpha}}=0.
    \end{equation}
    Now fix $z\in (1/2,3/4)$, the function 
   $$
 J(y):= \frac{(2z-y)^{\alpha}}{(2z-y)^{\alpha}+(1+y-2z)^{\alpha}}+\frac{y^{\alpha}}{y^{\alpha}+(1-y)^{\alpha}} , \quad y\in [2z-\frac{1}{2},1],
   $$
   has derivative $J^{\prime}(y)=-\alpha(f(2z-y)-f(y))$. Observe that $|2z-y-1/2|\leq |y-1/2|$ and thus $J^{\prime}(y)\leq 0$ (note that $f(x)$ decreases as $|x-1/2|$ increases). In particular, for any $y\in [2z-1/2,1]$, $J(y)\geq J(1)$, i.e.
   \begin{equation}
    \label{zyderremovey}
      \begin{aligned}
       &\quad -2z + \frac{(1-p)(2z-y)^{\alpha}}{(2z-y)^{\alpha}+(1+y-2z)^{\alpha}}+\frac{(1-p)y^{\alpha}}{y^{\alpha}+(1-y)^{\alpha}}+\frac{2pz^{\alpha}}{z^{\alpha}+(1-z)^{\alpha}} \\
       &\geq -2z + \frac{(1-p)(2z-1)^{\alpha}}{(2z-1)^{\alpha}+(2-2z)^{\alpha}}+(1-p)+\frac{2pz^{\alpha}}{z^{\alpha}+(1-z)^{\alpha}} \\
       &\stackrel{t=2z-1}{=}(2p-1)(\beta(t)+t-h(t))+\beta(t) > 0,
   \end{aligned}
   \end{equation}
 where $\beta(t)$ is defined in (\ref{equfhposi}) and we used Lemma \ref{betahpositve} in the last inequality. However, this contradicts (\ref{zsumequ}).
\end{proof}

Now we are ready to prove Proposition \ref{equifinite}.

\begin{proof}[Proof of Proposition \ref{equifinite}]
    It is direct to check that $\Lambda^{(\alpha)}_{p}$ includes $(0,0), (1/2,1/2),(1,1)$. The assertions (i) and (ii) follow from Lemma \ref{simplecaselem} and Lemma \ref{plarge12equili}, respectively.

By symmetry and Lemma \ref{simplecaselem}, to prove (iii) and (iv), we need to show that the set $$\Lambda^{(\alpha),+}_{p}:= \{(x,y)\in \Lambda^{(\alpha)}_{p}: 0<x<1/2<y<1, x+y> 1\}$$ 
is finite if $p>0$ and is empty if $p\geq (\alpha-1)/\alpha$.  \\
(iii) Assume that $p\geq (\alpha-1)/\alpha$ and $(\tilde{x},\tilde{y})\in \Lambda^{(\alpha),+}_{p}$, and in particular, $\tilde{z}:=(\tilde{x}+\tilde{y})/2\in (1/2,3/4)$. Recall $g$ and $h$ defined in (\ref{defghalpha}). We see from the proof of Lemma \ref{lemsoluualpha} that $g(\tilde{x})>0$ and 
\begin{equation}
 \label{xtildF1positive}
 F_{p,1}^{{\alpha}}(\tilde{x},\tilde{y})=g(\tilde{x})+p\left(\frac{\tilde{z}^{\alpha}}{\tilde{z}^{\alpha}+(1-\tilde{z})^{\alpha}}-\frac{1}{2}\right)>0,
\end{equation}
which contradicts our assumption, and proves (iii). \\
(iv) Assume that $p\in (0, (\alpha-1)/\alpha)$. Using arguments in the proof of Lemma \ref{lemsoluualpha}, we see that there exists $t_1\in (1/2,1)$ such that $g(u)$ is increasing on $[1/2,t_1]$ and decreasing on $[t_1,1]$, see e.g. Figure \ref{Fig225}. For $t,z\in [0,1]$, define 
$$\tilde{F}(t,z):=g(t)+p(h(z)-\frac{1}{2}).$$
  Since $g(t_1)>g(1/2)=0$ and $g(1)=-p/2$, there exists an $\varepsilon>0$ such that for any $z\in (1/2-\varepsilon,1]$, there exists a unique $y\in (t_1,1]$ such that $\tilde{F}(y,z)=0$. By the analytic implicit function theorem, see e.g. \cite[Theorem 6.1.2]{MR2977424}, this unique $y$ can be written as $y=q(z)$ where $q$ is analytic and increasing on $(1/2-\varepsilon,1)$ and 
\begin{equation}
    \label{yzder}
    q(\frac{1}{2})>t_1, \quad q(1)=1, \quad  q^{\prime}(z)=\frac{p\alpha f(z)}{1-\alpha(1-p)f(q(z))}=\frac{p\alpha f(z)}{-g^{\prime}(q(z))}>0.
    \end{equation}
    Observing that $f(z)$ decreases as $|z-1/2|$ increases, we see that $2-q^{\prime}(z)$ is increasing on $[1/2,1)$. Thus, using the convexity of $2z-q(z)$ and that $1-q(1/2)<1$, one has, for any $z\in [1/2,1)$, 
    \begin{equation}
        \label{bd2hzconvex}
        0<1-q(z)\leq 2z-q(z)<2-q(1)=1.
    \end{equation}
     Now we consider the analytic function $\widehat{F}(z):=\tilde{F}(2z-q(z),z)$ on $(1/2-\varepsilon,1)$. It is non-constant since its derivative equals
    \begin{equation}
        \label{g3derequ}
        \widehat{F}^{\prime}(z)= \left(\alpha(1-p)f(2z-q(z))-1\right)(2-q^{\prime}(z))-\alpha pf(z)
    \end{equation}
    which converges to $-2$ as $z\to 1$ (observe that $f(z) \to 0$ and $q^{\prime}(z) \to 0$ as $z\to 1$). As a non-constant analytic function, $\widehat{F}(z)$ has finite zeros on $[1/2,3/4]$. Observe that if $(\tilde{x},\tilde{y})\in \Lambda^{(\alpha),+}_{p}$ with $\tilde{z}:=(\tilde{x}+\tilde{y})/2> 1/2$, then $\tilde{y}=q(\tilde{z})$ and $\tilde{x}=2\tilde{z}-q(\tilde{z})$. In particular, $\tilde{z}$ is a zero of $\widehat{F}(z)$ on $[1/2,3/4]$. Since $\tilde{z}$ can take only finitely many values, (iv) is proved.
\end{proof}

For the proof of Proposition \ref{stableexistence}, we will need the following auxiliary lemmas.

\begin{lemma}
    \label{htalphat}
    Assume that $t\in (0,1/6]$. One has, \\
  (i) if $\alpha\geq 2$, then $2\alpha h(t) \leq t$; \\
  (ii) if $\alpha \in (1,2)$, then 
  $$
   h(t) \leq \frac{1}{\alpha}(1-\frac{1}{4}(\alpha-1))t.
  $$
\end{lemma}
\begin{proof}
 Fix $x=1/t\geq 6$, define 
 $$
\phi(u):= u \log (x-1) -\log (2u x-1), \quad u\in [2,\infty),
 $$
 and 
 $$
 \varphi(u):=u\log (x-1) -\log (\frac{u}{1-\frac{1}{4}(u-1)} x-1), \quad u\in [1,2].
  $$
 Then $\phi(2)>0$ and $\varphi(1)=0$. Moreover, 
$$
    \phi^{\prime}(u)\geq   \log (x-1)-\frac{2x}{4x-1} >0, \quad u\in [2,\infty),
    $$
    and for $u\in [1,2]$,
    $$
    \begin{aligned}
      \varphi^{\prime}(u) &= \log (x-1)-\frac{5}{5-u}\frac{x}{ux-1+\frac{1}{4}(u-1)}\geq \log (x-1)-\frac{5}{5-u}\frac{x}{ux-1} \\
      &\geq \log (x-1)-\frac{5}{4}\frac{x}{x-1} \geq \log 5 -\frac{3}{2}>0,
    \end{aligned}
    $$
    where we used that $(5-u)(ux-1)$ achieves its minimum on $[1,2]$ at $u=1$. Therefore, $\phi(u)>0$ on $[2,\infty)$ and $\varphi(u)>0$ on $(1,2]$. It remains to notice that (i) and (ii) follow from that $\phi(\alpha)>0$ and that $\varphi(\alpha)>0$, respectively.
\end{proof}

\begin{lemma}
    \label{ualphavalpha}
    For $\alpha>1$ and $p\in (0,\alpha^{-1}\min((\alpha-1)/6,1/3)]$, let
    \begin{equation}
        \label{valphap}
        v_{\alpha, p} := \frac{1}{2}-\frac{1}{2}\sqrt{\frac{(\alpha-1)(1-p)}{\alpha-1+p+\alpha p-\alpha p^2}}=\frac{1}{2}-\frac{1}{2}\sqrt{1-\frac{\alpha p(2-p)}{\alpha-1+p+\alpha p-\alpha p^2}}.
    \end{equation}
    Then $g(v_{\alpha, p})<0$, and in particular, $u_{\alpha, p}<v_{\alpha, p}$.
\end{lemma}
\begin{proof}
For $x\in (0,1)$, one has $1-x<\sqrt{1-x}<1-x/2$, and in particular,
\begin{equation}
    \label{valphapbd}
    \frac{\alpha p(2-p)}{4(\alpha-1+p+\alpha p-\alpha p^2)} <v_{\alpha, p} <\frac{\alpha p(2-p)}{2(\alpha-1+p+\alpha p-\alpha p^2)}<\frac{1}{6},
\end{equation}
where the last inequality follows from that $p\leq  \alpha^{-1}(\alpha-1)/6$. Observe that 
    $$
g(t)=-(1-p)(t-h(t))+p(\frac{1}{2}-t).
    $$
 Then $g(v_{\alpha, p})<0$ if and only if
    \begin{equation}
        \label{alphalarge2vh}
        \frac{1}{p} (v_{\alpha, p}-h(v_{\alpha, p})) > \frac{1}{1-p}(\frac{1}{2}-v_{\alpha, p})=\frac{1}{2(1-p)}\sqrt{\frac{(\alpha-1)(1-p)}{\alpha-1+p+\alpha p-\alpha p^2}}.
    \end{equation}
(i) If $\alpha\geq 2$, by Lemma \ref{htalphat} (i) and (\ref{valphapbd}), 
$$
\frac{1}{p} (v_{\alpha, p}-h(v_{\alpha, p}))> (1-\frac{1}{2\alpha})  \frac{\alpha (2-p)}{4(\alpha-1+p+\alpha p-\alpha p^2)}> \frac{(\alpha-\frac{1}{2}) \sqrt{1-p}}{2(\alpha-1+p+\alpha p-\alpha p^2)},
$$
where we used that $1-p/2> \sqrt{1-p}$. To prove (\ref{alphalarge2vh}), it suffices to show that 
    $$
   (\alpha-\frac{1}{2}) (1-p) - (\alpha-1)\sqrt{1+\frac{p(\alpha+1)}{\alpha-1}}\geq 0.
    $$
 Using that $\sqrt{1+x}\leq 1+x/2$ for $x\geq 0$, we see the left-hand side is lower bounded by
 $$(\alpha-\frac{1}{2}) (1-p) - (\alpha-1+\frac{p(\alpha+1)}{2})=\frac{1-3\alpha p}{2}\geq 0,$$
 which completes the proof for the case $\alpha\geq 2$. \\
(ii) If $\alpha \in (1,2)$, similarly, by Lemma \ref{htalphat} (ii) and (\ref{valphapbd}), it suffices to show that
$$
\frac{5(\alpha-1)(1-p)}{4} - (\alpha-1+\frac{p(\alpha+1)}{2}) \geq 0.
$$
Using that $p\leq \alpha^{-1}(\alpha-1)/6$, we see that the left-hand side (LHS) satisfies
$$
\begin{aligned}
   \text{LHS}= \frac{\alpha-1}{4} - (\frac{5(\alpha-1)}{4}+\frac{\alpha+1}{2})p &\geq \frac{(\alpha-1)}{6\alpha} (\frac{3\alpha}{2} - \frac{5(\alpha-1)}{4}-\frac{\alpha+1}{2}) \\
   &=\frac{(\alpha-1)}{24\alpha} (3-\alpha) >0.
\end{aligned}
$$

Thus, in either case, (\ref{alphalarge2vh}) holds, and equivalently, $g(v_{\alpha, p})<0$. Since $p\leq \alpha^{-1}(\alpha-1)/6$, from the proof of Lemma \ref{lemsoluualpha}, we see that $u_{\alpha, p}$ exists and $u_{\alpha, p}<v_{\alpha, p}<1/2$.
\end{proof}

We now prove Proposition \ref{stableexistence}.

\begin{proof}[Proof of Proposition \ref{stableexistence}]
   (i) Notice that $f(1/2)=1$ and $f(u_{\alpha, p})=f(1-u_{\alpha, p})$. By (\ref{lambdapm}), we need to prove that 
\begin{equation}
    \label{lambdaumnega}
    -1+\alpha p+\alpha(1-p)f(u_{\alpha, p})<0.
\end{equation}
The case $p=0$ is trivial since $u_{\alpha, 0}=0$. For $p>0$, one has $u_{\alpha, p}>0$. Moreover, using that $g(u_{\alpha, p})=0$ (and thus $g(1-u_{\alpha, p})=0$), we have
$$
\frac{u_{\alpha, p}^{\alpha-1}}{u_{\alpha, p}^{\alpha}+(1-u_{\alpha, p})^{\alpha}}=\frac{u_{\alpha, p}-\frac{p}{2}}{(1-p)u_{\alpha, p}}, \quad \frac{(1-u_{\alpha, p})^{\alpha-1}}{u_{\alpha, p}^{\alpha}+(1-u_{\alpha, p})^{\alpha}}=\frac{1-u_{\alpha, p}-\frac{p}{2}}{(1-p)(1-u_{\alpha, p})}.
$$
Therefore, using (\ref{deffderxm}), we can write
$$
f(u_{\alpha, p})=\frac{(u_{\alpha, p}-\frac{p}{2})(1-u_{\alpha, p}-\frac{p}{2})}{(1-p)^2u_{\alpha, p}(1-u_{\alpha, p})}.
$$
Then (\ref{lambdaumnega}) is equivalent to 
$$
u_{\alpha, p} < \frac{1}{2}-\frac{1}{2}\sqrt{\frac{(\alpha-1)(1-p)}{\alpha-1+p+\alpha p-\alpha p^2}},
$$
which follows from Lemma \ref{ualphavalpha}. \\
(ii) Recall $F^{(\alpha)}_p$ defined in (\ref{Falphapdef}). If $\alpha\geq 3$ and $p=1/2-\alpha^{-1}\log \alpha$, then for any $y\in [0,1]$, one has  
$$
\begin{aligned}
    F_{p,1}^{(\alpha)}(\frac{1}{2}-\frac{\log \alpha}{2\alpha},y) &< -\frac{1}{2}+\frac{\log \alpha}{2\alpha}+(\frac{1}{2}+\frac{\log \alpha}{\alpha} )(\frac{1-\alpha^{-1}\log \alpha}{1+\alpha^{-1}\log \alpha})^{\alpha}+\frac{1}{2}-\frac{\log \alpha}{\alpha} \\
    &\leq - \frac{\log \alpha}{2\alpha}+(\frac{1}{2}+\frac{\log 3}{3} )(1-\frac{6\log \alpha}{(3+\log 3) \alpha})^{\alpha} \\
    &\leq -\frac{1}{2}\alpha^{-\frac{6}{3+\log 3}} \left(\alpha^{\frac{3-\log 3}{3+\log 3}} \log \alpha - 1-\frac{2\log 3}{3}\right) <0,
\end{aligned}
$$
where we used that $1-t\leq e^{-t}$ for $t\in [0,1]$ and that 
$$\frac{\log \alpha}{\alpha} \leq \frac{\log 3}{3}, \quad \alpha^{\frac{3-\log 3}{3+\log 3}} \log \alpha - 1-\frac{2\log 3}{3}\geq 3^{\frac{3-\log 3}{3+\log 3}} \log 3 - 1-\frac{2\log 3}{3}>0.$$
If $y\in [1/2,1]$, then for any $p\leq 1/2-\alpha^{-1}\log \alpha$, one still has
$$
F_{p,1}^{(\alpha)}(\frac{1}{2}-\frac{\log \alpha}{2\alpha},y)<0,
$$
since the left-hand side is an increasing function in $p$. Moreover, for any $x\in [0,1/2]$ and $p\leq 1/2-\alpha^{-1}\log \alpha$, one has
$$
F_{p,2}^{(\alpha)}(x,\frac{1}{2}+\frac{\log \alpha}{2\alpha})=-F_{p,1}^{(\alpha)}(\frac{1}{2}-\frac{\log \alpha}{2\alpha},1-x)>0.
$$
On the other hand, since $p>0$, one has $F_{p,1}^{(\alpha)}(0,y)>0$ and $F_{p,2}^{(\alpha)}(x,1)<0$ for $x\in [0,1/2]$ and $y\in [1/2,1]$. Therefore, the maximum of $L^{(\alpha)}_{p}$ on the square $[0,1/2-(2\alpha)^{-1}\log \alpha]\times [1/2+(2\alpha)^{-1}\log \alpha,1]$ can only be achieved at some interior point, which is asymptotically stable.

(iii) We can assume that $\alpha\geq 2$ and thus $p<(\alpha-1)/\alpha$. We shall use the functions 
$$\tilde{F}(t,z)=g(t)+p(h(z)-\frac{1}{2})$$
and $q(z)$ defined in the proof of Proposition \ref{equifinite} (iv): For any $z\in [1/2,1]$, we have $\tilde{F}(q(z),z)=0$ and $q(z)\in (t_1,1]$ where $t_1\in (1/2,1)$ is such that $g$ is strictly increasing on $[1/2, t_1]$ and strictly decreasing on $[t_1,1]$. By (\ref{zmlargez}), if $z \in (1/2,1)$, then  
$$
g(q(z))=-p(h(z)-\frac{1}{2}) < g(z),
$$
whence we have $q(z)>z$ for $z\in (1/2,1)$. Fix $\delta>0$ such that  $2\delta<\min(1/2-p,p)$, then for any $z\in [1/2+\delta, 1]$, 
$$
q(z) =(1-p)h(q(z))+ ph(z)> h(z) \geq h(1/2+\delta)=\frac{(1/2+\delta)^{\alpha}}{(1/2+\delta)^{\alpha}+(1/2-\delta)^{\alpha}} \to 1,
$$
as $\alpha \to \infty$. In particular, we can choose a large $\alpha$ such that $q(z)>1-\delta$ for all $z\geq 1/2+\delta$. Therefore, for any $z\in [1/2+\delta, 3/4-\delta]$, 
\begin{equation}
  \label{xaway122zhz}
  0< 2z-q(z) < 2z-1+\delta \leq \frac{1}{2}-\delta.
\end{equation}
Note that 
\begin{equation}
    \label{mfto0uni}
    \alpha f(t)=\frac{\alpha t^{\alpha-1}(1-t)^{\alpha-1}}{(t^{\alpha}+(1-t)^{\alpha})^2}  \to 0, \quad \text{as}\ \alpha\to \infty,
  \end{equation}
  uniformly on $[0,1/2- \delta]\cup [1/2+ \delta,1]$. Recall that $\widehat{F}(z)=\tilde{F}(2z-q(z),z)$. Using (\ref{yzder}), (\ref{g3derequ}) and (\ref{xaway122zhz}), we have
\begin{equation}
  \label{Ftildederlim}
  \lim_{\alpha \to \infty} \widehat{F}^{\prime}(z) =-2, \quad \text{uniformly on}\  [\frac{1}{2}+\delta, \frac{3}{4}-\delta].
\end{equation}
Moreover, by the choice of $\delta$, one has
\begin{equation}
    \label{Fhatendvalues}
    \lim_{\alpha\to \infty}\widehat{F}(\frac{1}{2}+\delta)=p-2\delta> 0, \quad \lim_{\alpha\to \infty}\widehat{F}(\frac{3}{4}-\delta)=p-\frac{1}{2}+2\delta <0.
\end{equation}
Therefore, for all large $\alpha$, we have $\widehat{F}(1/2+\delta)>0$ and $\widehat{F}(3/4-\delta)<0$, and by (\ref{Ftildederlim}), there exists a unique $z_{\alpha} \in (1/2+\delta, 3/4-\delta)$ such that $\widehat{F}(z_{\alpha})=0$. In particular, $s_{\alpha}:=(2z_{\alpha}-q(z_{\alpha}),q(z_{\alpha}))$ is an equilibrium. Observe that by (\ref{Ftildederlim}) and (\ref{Fhatendvalues}),
$$
\lim_{\alpha\to \infty} z_{\alpha}= \frac{3}{4}-\delta- \frac{1}{2}(\frac{1}{2}-2\delta-p)=\frac{1+p}{2},
$$
and thus $\lim_{\alpha\to \infty} s_{\alpha} =(p,1)$. Since none of $2z_{\alpha}-q(z_{\alpha}),z_{\alpha}$ and $q(z_{\alpha})$ are in the interval $[1/2-\delta,1/2+\delta]$, by (\ref{mfto0uni}), we have
$$\lim_{\alpha\to \infty}\left(\alpha f(2z_{\alpha}-q(z_{\alpha})), \alpha f(z_{\alpha}), \alpha f(q(z_{\alpha})) \right)=(0,0,0),
$$ 
which, by (\ref{lambdapm}), implies that $\lim_{\alpha\to \infty}\lambda_+(s_{\alpha}) =-1$.
Thus, $\lambda_+(s_{\alpha})<0$ for all large $\alpha$.
\end{proof}

\section{Stochastic approximation algorithm} 
\label{stocappsec}

\begin{proof}[Proof of Proposition \ref{propstocappxy}]
We assume that $W(n)$ is a polynomial of degree $\alpha$ as in (\ref{defWpolyalpha}). The case of power function reinforcement can be proved similarly. Note that 
$$
\frac{W(B_n(1))}{n^{\alpha}}=x_n+O(\frac{1}{n}), \quad \frac{W(B_n^*)}{n^{\alpha}}=x_n+y_n+O(\frac{1}{n}),
$$
thus, by (\ref{defxi}), for any $n\in \NN$, 
    \begin{equation}
        \label{xinGncond}
          \begin{aligned}
            \EE(\xi_{n+1}(1)\mid \GG_n) &= \frac{(1-p)W(B_n(1))}{W(B_n(1))+W(R_n(1))}+\frac{p W(B_n^*)}{W(B_n^*)+W(R_n^*)} \\
     &=\frac{(1-p)x_n^{\alpha}+O(\frac{1}{n})}{x_n^{\alpha}+(1-x_n)^{\alpha}+O(\frac{1}{n})}+\frac{p(x_n+y_n)^{\alpha}+O(\frac{1}{n})}{(x_n+y_n)^{\alpha}+(2-x-y)^{\alpha}+O(\frac{1}{n})}\\
      &=\frac{(1-p)x_n^{\alpha}}{x_n^{\alpha}+(1-x_n)^{\alpha}}+\frac{p(x_n+y_n)^{\alpha}}{(x_n+y_n)^{\alpha}+(2-x_n-y_n)^{\alpha}}+O(\frac{1}{n}).
  \end{aligned}
    \end{equation}
  By (\ref{defBR}), we have 
$$
    x_{n+1}-x_n=\frac{B_{n+1}(1)-B_n(1)-x_n}{n+1+B_0(1)+R_0(1)} =\frac{1}{n+1}(F^{(\alpha)}_{p,1}(x_n,y_n)+\varepsilon_{n+1}(1)+r_{n+1}(1)),
$$
where
\begin{equation}
    \label{defvarep1nMart}
 \varepsilon_{n+1}(1):=\frac{n+1}{n+1+B_0(1)+R_0(1)}\left(\xi_{n+1}(1)-\EE(\xi_{n+1}(1)\mid \GG_n)\right) ,
\end{equation}
and, by (\ref{xinGncond}), 
$$
r_{n+1}(1):=(n+1)\frac{-x_n+\EE(\xi_{n+1}(1)\mid \GG_n)}{n+1+B_0(1)+R_0(1)}-F^{(\alpha)}_{p,1}(x_n,y_n)=O(\frac{1}{n}).
$$
The equation for $y_{n+1}-y_{n}$ is proved similarly. 
\end{proof}

\begin{lemma}
    \label{lemzn}
    Recall $(z_n)_{n\in \NN}$ defined by (\ref{zndefpropB}). Under $\PP_p^{(1)}$, one has: \\
    (i) The process $(z_n)_{n\in \NN}$ satisfies the following recursion:
    $$
    z_{n+1}-z_n =\frac{1}{2n+2+B_0^*+R_0^*}\left(\tilde{\varepsilon}_{n+1}+\tilde{r}_{n+1}\right), \quad n\in \NN,
     $$
     where $(\tilde{\varepsilon}_{n+1})_{n\in \NN}$ and $(\tilde{r}_{n+1})_{n\in \NN}$ are adapted sequence such that for all $n\in \NN$, 
     \begin{equation}
        \label{tildevarrbd}
        \EE (\tilde{\varepsilon}_{n+1}\mid \GG_n)=0, \ \EE (\tilde{\varepsilon}_{n+1}^2\mid \GG_n) \leq (2+\frac{C}{2n+B_0^*+R_0^*})z_n, \ |r_{n+1}| \leq \frac{Cz_n}{2n+B_0^*+R_0^*},
     \end{equation}
      where $C$ is a positive constant. In particular, $(z_n)_{n\in \NN}$ converges a.s.\\
    (ii) For $k\geq 1$, let $\tau(k):=\inf\{n\geq k: z_n \geq 2z_k\}$ with the convention that $\inf \emptyset=\infty$. Then, there exists a positive integer $K$ such that for all $k\geq K$,
       $$
      \PP(\{\tau(k) <\infty\}\bigcup\{\lim_{n\to \infty}z_n=0, \tau(k)=\infty\}\mid \GG_k) \leq \frac{4}{B_k^*}.
       $$
     \end{lemma}
     \begin{proof}
        (i) For $n\in \NN$, let 
$$\tilde{\varepsilon}_{n+1}:= \xi_{n+1}(1)+\xi_{n+1}(2)-\EE(\xi_{n+1}(1)+\xi_{n+1}(2)\mid \GG_n)$$
and 
$$
r_{n+1}:=\EE(\xi_{n+1}(1)+\xi_{n+1}(2)\mid \GG_n)-2z_n=(1-p)(x_n+y_n-2z_n).
$$
Then
$$
    z_{n+1}-z_n=\frac{\xi_{n+1}(1)+\xi_{n+1}(2)-2z_n}{2n+2+B_0^*+R_0^*} =\frac{\tilde{\varepsilon}_{n+1}+\tilde{r}_{n+1}}{2n+2+B_0^*+R_0^*}.
$$
By definition, 
  \begin{equation}
    \label{xnyn2zn}
    x_n+y_n-2z_n= \frac{B_0(2)+R_0(2)-B_0(1)-R_0(1)}{2n+B_0^*+R_0^*}(x_n-y_n).
  \end{equation}
Using that $x_n\leq (2n+B_0^*+R_0^*)z_n/(n+B_0(1)+R_0(1))$ and $y_n\leq (2n+B_0^*+R_0^*)z_n/(n+B_0(2)+R_0(2))$, we have 
  \begin{equation}
    \label{bdrnequ}
    |r_{n+1}|=(1-p)|x_n+y_n-2z_n| \leq \frac{Cz_n}{2n+B_0^*+R_0^*}.
  \end{equation}
Note that conditional on $\GG_n$, the two random variables $\xi_{n+1}(1)$ and $\xi_{n+1}(2)$ are independent Bernoulli random variables. Thus, $\EE (\tilde{\varepsilon}_{n+1}\mid \GG_n)=0$ and 
    $$
        \EE (\tilde{\varepsilon}_{n+1}^2\mid \GG_n) \leq \EE(\xi_{n+1}(1)+\xi_{n+1}(2)\mid \GG_n) =2z_n +r_{n+1},
    $$
    which implies (\ref{tildevarrbd}) in virtue of (\ref{bdrnequ}). Now let $M_0:=0$ and
$$
M_{n}:=\sum_{j=1}^{n} \frac{\tilde{\varepsilon}_{j}}{2j+B_0^*+R_0^*},\quad n \geq 1.
$$
 Then, by (\ref{tildevarrbd}), the process $(M_n)_{n\in \NN}$ is a $L^2$-bounded martingale, and thus converges a.s. Moreover, 
 $$
 \sum_{n=1}^{\infty}\frac{|\tilde{r}_{n}|}{2n+B_0^*+R_0^*} <\infty.
 $$
These show that $(z_n)_{n\in \NN}$ converges a.s.. 
       
       (ii) For any $k\geq 1$, by (\ref{tildevarrbd}),
       $$
       \sum_{j=k+1}^{\tau(k)} \frac{|\tilde{r}_{j}|}{2j+B_0^*+R_0^*} \leq  \sum_{j=k+1}^{\tau(k)} \frac{2Cz_k}{(2j+B_0^*+R_0^*)(2j+B_0^*+R_0^*-2)} \leq \frac{Cz_k}{2k+B_0^*+R_0^*},
       $$
     and, similarly, the quadratic variation of $(M_n)_{n\in \NN}$ satisfies
     $$
     \langle M \rangle_{\tau(k)}-\langle M \rangle_{k} \leq \sum_{j=k+1}^{\tau(k)} \frac{\EE (\tilde{\varepsilon}_{j}^2\mid \GG_{j-1})}{(2j+B_0^*+R_0^*)^2} \leq  \left(2+\frac{C}{2k+B_0^*+R_0^*}\right)\frac{z_k}{2k+B_0^*+R_0^*}.
     $$
     Choose $K \in \NN$ such that $C/(2K+B_0^*+R_0^*)<1/4$. Then, for all $k\geq K$, by the optional stopping theorem, one has
       $$\frac{9z_k^2}{16}\PP(E_k \mid \GG_k) \leq \EE \left((M_{\tau(k)}-M_k)^2\mathds{1}_{E_k}\mid \GG_k\right) \leq \EE  (\langle M \rangle_{\tau(k)} -\langle M \rangle_{k}\mid \GG_k) \leq \frac{9}{4}\frac{z_k}{2k+B_0^*+R_0^*},$$
    where $E_k:=\{\tau(k) <\infty\}\cup\{\lim_{n\to \infty}z_n=0, \tau(k)=\infty\}$
 and we used that on the event $E_k$, 
       $$
|M_{\tau(k)}-M_k| \geq |z_{\tau(k)}-z_k|- \sum_{j=k+1}^{\tau(k)} \frac{|\tilde{r}_{j}|}{2j+B_0^*+R_0^*} \geq z_k-\frac{z_k}{4}=\frac{3z_k}{4}.
       $$
This proves (ii) since $z_k(2k+B_0^*+R_0^*)=B_k^*$.  
     \end{proof}

We now set $t_0:=0$, $t_n:=\sum_{i=1}^n1/i$ and define an interpolated process $(I(t))_{t\geq 0}$ by 
$$
I(t_n+s)=(x_n,y_n)+s\frac{(x_{n+1},y_{n+1})-(x_n,y_n)}{t_{n+1}-t_n},\quad n\in \NN, s\in [0,\frac{1}{n+1}].
$$
By Proposition \ref{propstocappxy} and \cite[Proposition 4.2, Remark 4.5]{MR1767993}, the interpolated process $I$ is an asymptotic pseudotrajectory of the flow induced by the vector field $F^{(\alpha)}_p$. We now prove Theorem \ref{xnynu}.

  \begin{proof}[Proof of Theorem \ref{xnynu}]
We first prove the a.s.-convergence of $((x_n,y_n))_{n\in \NN}$. We first assume that $\alpha=1$. The case $p=0$ follows from the classical results for the P\'olya urn model. If $p>0$, recall that $\Lambda^{(1)}_p=\{(x,x): x\in [0,1]\}$ and $L^{(1)}_p=-p(x-y)^2/4$ by Example \ref{examLalpha}. Since $[0,1]^2$ is compact, by \cite[Theorem 5.7]{MR1767993}, the limit set of the interpolated process $I$
$$
L(I):=\bigcap_{t \geq 0} \overline{I([t, \infty))}
$$
is internally chain transitive a.s.. Then, by Proposition \ref{stogradsys} and \cite[Proposition 6.4]{MR1767993}, almost surely, $L(I)\subset \Lambda^{(1)}_p$, and thus, the sequence $(x_n - y_n)_{n\in \NN}$ converges to $0$ a.s.. On the other hand, by Lemma \ref{lemzn} (i) and (\ref{xnyn2zn}), $(x_n + y_n)_{n\in \NN}$ converges a.s., and in particular, $((x_n,y_n))_{n\in \NN}$ converges a.s.. For $\alpha>1$, the proof is similar: By Proposition \ref{equifinite}, there are only finitely many equilibria of the gradient system (\ref{odedxFm}). Then one can directly apply \cite[Corollary 6.6]{MR1767993} to conclude that $((x_n,y_n))_{n\in \NN}$ converges a.s. to an equilibrium.

(i) We have shown that for $p>0$, almost surely, $(x_n)_{n\in \NN}$ and $(y_n)_{n\in \NN}$ have the same limit. Let $K$ be as in Lemma \ref{lemzn} (ii), and let $\theta_K:=\inf\{k\geq K: B_k^*>8 \}$. Then $\theta_K$ is a.s. finite. By Lemma \ref{lemzn} (ii), for any $j\in \NN$,
$$
\PP(\lim_{n\to \infty}z_n=0\mid \GG_{\theta_K+j})\leq \frac{4}{B^*_{\theta_K+j}} \leq \frac{1}{2},
$$
which implies that $\PP(\lim_{n\to \infty}z_n=0)=0$ by Levy's 0-1 law. Similarly, one can show that $\PP(\lim_{n\to \infty}z_n=1)=0$. Thus, $\PP_p^{(1)}(\mathcal{D})=0$.

(ii) We assume that $p\leq \alpha^{-1}\min((\alpha-1)/6,1/3)$. The existence of $u_{\alpha, p}$ follows from Lemma \ref{lemsoluualpha}. By Proposition \ref{stableexistence}, the equilibrium $(u_{\alpha, p},1-u_{\alpha, p})$ is asymptotically stable. Moreover, for any open neighborhood $\mathcal{N}$ of $(u_{\alpha, p},1-u_{\alpha, p})$ and any $m\in \NN$, one has $\PP((x_n,y_n)\in \mathcal{N}\ \text{for some } n\geq m)>0$. Then one can deduce (ii) from \cite[Theorem 7.3]{MR1767993}.       
       \end{proof}

 The following auxiliary lemma will be used in the proof of Theorem \ref{pmlarger12}.  

\begin{lemma}
    \label{unconunstable}
    If $(x,y) \in \Lambda^{(\alpha)}_{p}$ is unstable, i.e. $\lambda_+(x,y)>0$, then
        \begin{equation}
         \label{unsprob0equ}
     \PP_p^{(\alpha)}(\lim_{n\to \infty}(x_n,y_n)= (x,y))=0.
        \end{equation}
\end{lemma}
\begin{proof}
    Since $(x,y)$ is unstable, by Proposition \ref{equifinite} (iv), it is not on the boundary of $[0,1]^2$, and thus, there exists a neighborhood $\mathcal{N}$ of $(x,y)$ such that any $(u,v)\in \mathcal{N}$ is bounded away from the boundary. Now we show that there exists a constant $b>0$ such that for any $\theta \in [0,2\pi]$,
       \begin{equation}
        \label{nonconpemantle}
        \EE (\left(\varepsilon_{n+1}(1)  \cos\theta +\varepsilon_{n+1}(2)  \sin\theta \right)^+ \mid\mathcal{G}_n)\mathds{1}_{\{(x_n,y_n)\in \mathcal{N}\}}\geq b \mathds{1}_{\{(x_n,y_n)\in \mathcal{N}\}},
       \end{equation}
       where $(\varepsilon_{n+1})_{n\in \NN}$ is given in Proposition \ref{propstocappxy} and $x^+:=\max(x,0)$. By (\ref{defxi}) and (\ref{defvarep1nMart}), we can find positive constants $C_1,C_2$ such that 
       $$\PP(\varepsilon_{n+1}(1) \geq C_2, \varepsilon_{n+1}(2) \geq C_2\mid \GG_n)\mathds{1}_{\{(x_n,y_n) \in \mathcal{N}\}}\geq C_1 \mathds{1}_{\{(x_n,y_n)\in \mathcal{N}\}}.$$
     Therefore, for any $\theta \in [0,\pi/2]$, the left-hand side of (\ref{nonconpemantle}) is lower bounded by 
       $$
   C_1 C_2(\cos \theta + \sin\theta)\mathds{1}_{\{(x_n,y_n)\in \mathcal{N}\}} 
        \geq  C_1 C_2\mathds{1}_{\{(x_n,y_n)\in \mathcal{N}\}}.
       $$
    The cases $\theta \in [\pi/2,\pi], [\pi,3\pi/2], [3\pi/2,2\pi]$ can be proved similarly. 
    
    Now observe that $F^{(\alpha)}_{p}$ is $C^{\infty}$. Then, (\ref{unsprob0equ}) follows from \cite[Theorem 2.2.4]{raimond2023non}, see also \cite[Theorem 1]{MR1055428}.
    \end{proof}

 \begin{proof}[Proof of Theorem \ref{pmlarger12}]
        (i) If $p\geq \min(1/2,\alpha^{-1}(\alpha-1))$, then, by Proposition \ref{equifinite} and Corollary \ref{pgeq12unstalbe}, $\Lambda^{(\alpha)}_{p}\backslash\{(0,0),(1,1)\}$ consists of finitely many unstable equilibria. From the proof of Theorem \ref{xnynu}, we see that $(x_n,y_n)$ converges a.s. to an equilibrium. Lemma \ref{unconunstable} then implies that $\PP_p^{(\alpha)}(\mathcal{D})=1$, and thus, $p_{\alpha}\leq \min(1/2,\alpha^{-1}(\alpha-1))$. 
        
        We assume that $p_{\alpha}=\min(1/2,\alpha^{-1}(\alpha-1))$ for some $\alpha>1$. In particular, for any $q<\min(1/2,\alpha^{-1}(\alpha-1))$, there exists $p\in (q,\min(1/2,\alpha^{-1}(\alpha-1)))$ such that $\PP_p^{(\alpha)}(\mathcal{D})<1$. By Proposition \ref{equifinite} and Lemma \ref{unconunstable}, we can find two sequences $(p^{(n)})_{n\geq 1}$ and $(\tilde{s}_n)_{n\geq 1}$ such that for any $n\geq 1$, 
$$0<p^{(n)}<\min(\frac{1}{2},\frac{\alpha-1}{\alpha}),\quad \tilde{s}_n\in \Lambda^{(\alpha)}_{p} \bigcap  \left([0,\frac{1}{2}] \times [\frac{1}{2},1]\right)\ \text{with} \ \lambda_+(\tilde{s}_n)\leq 0,$$ 
and
 $$\lim_{n\to \infty}p^{(n)} = \min(\frac{1}{2},\frac{\alpha-1}{\alpha}),\quad \PP_{p^{(n)}}^{(\alpha)}(\lim_{n\to \infty}(x_n,y_n)=\tilde{s}_n)>0.$$
 Since $[0,1]^2$ is compact, by possibly choosing a subsequence, we may assume that 
 $
 \lim_{n\to \infty} \tilde{s}_n=\tilde{s}
 $
for some $\tilde{s} \in [0,1/2] \times [1/2,1]$. Since $F^{(\alpha)}_{p}(x,y)$ and $\lambda_+(x,y)$ are continuous function in $(p,x,y)$, we see that $F^{(\alpha)}_{p}(\tilde{s})=0$ and $\lambda_+(\tilde{s})\leq 0$ with $p=\min(1/2,\alpha^{-1}(\alpha-1))$, which contradicts Corollary \ref{pgeq12unstalbe}.

As shown in the proof of Theorem \ref{xnynu} (ii), if $\alpha>1$, then $\PP_p^{(\alpha)}(\lim_{n\to \infty}(x_n,y_n)= (x,y))>0$ for any $(x,y) \in \mathcal{E}_{p}^{(\alpha)}$. Thus, (ii) and (iii) are corollaries of Proposition \ref{stableexistence} (ii) and (iii) (note that for $p=0$, by (\ref{lambdapm}) and Proposition \ref{equifinite}, one has $(0,1) \in \mathcal{E}_{p}^{(\alpha)}$).
 \end{proof}

\section{Continuous-time construction with time-delays}
\label{ctdelays}

Let $X$ be the jump process defined in Section \ref{constcondelay} and $(\FF_t)_{t\geq 0}$ be its natural filtration. By Proposition \ref{ZkdNin}, we may define $X$ and the IUM $(N_n)_{n\in \NN}$ on the same probability space $\PS$ such that (\ref{ZkdNki}) holds.

The following lemma will be used in the proof of Theorem \ref{wnnewprop}. Recall that $A_n=\sum_{i=n}^{\infty}1/W(i)^2$.

\begin{lemma}
    \label{martdeviation}
    Assume that  $\{W(n)\}_{n\in \NN}$ satisfies (\ref{strongreinforce}) and $\lim_{n\to \infty} W(n) \sqrt{A_n} =\infty$. Let $\{\theta^{(1)}_n\}_{n\geq 1}$ and $\{\theta^{(2)}_n\}_{n\geq 1}$ be independent Exp(1)-distributed random variables, and for $k\geq 1$, let
    $$
S^{(k)}:=\sum_{n=k}^{\infty}\frac{\theta_n^{(1)}-\theta_n^{(2)}}{W(n)}; \quad S^{(k)}_m:=\sum_{n=k}^{m}\frac{\theta_n^{(1)}-\theta_n^{(2)}}{W(n)}, \ m\geq k.
    $$
    Then, there exists $K\in \NN_+$ such that for all large $k\geq K$,
    $$
\PP(S^{(k)} > \frac{1}{4} \sqrt{A_k } )\geq \frac{5}{32}.
    $$
\end{lemma}
\begin{proof}
    Our assumptions imply that there exists $K\in \NN_+$ such that $W(n)\sqrt{A_n}\geq 64 \EE |\theta_{n}^{(1)}-\theta_{n}^{(2)}|^3$ for all $n\geq K$. Now fix $k\geq K$, let $\tau:=\inf\{m\geq k: |S^{(k)}_m|>  \sqrt{A_k }/4 \}$ and $T_{\theta}:=\inf\{m\geq k: |\theta_m^{(1)}-\theta_m^{(2)}|>  W(m)\sqrt{A_k}/2 \}$. Note that $\tau\leq T_{\theta}$. By definition, $ \EE (S^{(k)}_{\tau})^2\mathds{1}_{\{\tau=\infty\}} \leq A_k/16$. Moreover, 
    $$
        \EE (S^{(k)}_{\tau})^2\mathds{1}_{\{\tau<\infty\}}=\EE (S^{(k)}_{\tau-1}+\frac{\theta_{\tau}^{(1)}-\theta_{\tau}^{(2)}}{W(\tau)})^2\mathds{1}_{\{\tau<\infty\}} \leq \frac{A_k}{8}  +2 \EE \left(\frac{\theta_{\tau}^{(1)}-\theta_{\tau}^{(2)}}{W(\tau)}\right)^2.
    $$
By the choice of $K$, one has
$$
\begin{aligned}
    \EE \left(\frac{\theta_{\tau}^{(1)}-\theta_{\tau}^{(2)}}{W(\tau)}\right)^2 &= \EE \left(\frac{\theta_{\tau}^{(1)}-\theta_{\tau}^{(2)}}{W(\tau)}\right)^2\mathds{1}_{\{\tau < T_{\theta}\}}+\sum_{n=k}^{\infty} \EE \left(\frac{\theta_{n}^{(1)}-\theta_{n}^{(2)}}{W(n)}\right)^2\mathds{1}_{\{T_{\theta}=n\}} \\
    &\leq \frac{A_k }{4} + \sum_{n=k}^{\infty} \frac{2}{W(n)^2}\frac{\EE |\theta_{n}^{(1)}-\theta_{n}^{(2)}|^3}{ W(n)\sqrt{A_k}} \leq \frac{9A_k }{32}.
\end{aligned}
$$
Applying the optional stopping theorem to the $L^2$-bounded martingale $(S^{(k)}_m)_{m\geq k}$, we have
$$
\frac{3A_k}{4}\geq  \EE (S^{(k)}_{\tau})^2 = \EE \langle S^{(k)} \rangle_{\tau} \geq \EE \langle S^{(k)} \rangle_{\infty} \mathds{1}_{\{\tau=\infty\}} =  2 A_k\PP(\tau=\infty),
$$
where the quadratic variation of the martingale $(S^{(k)}_m)_{m\geq k}$ is given by
$$
\langle S^{(k)} \rangle_{m} = \sum_{n=k}^{m}\frac{2}{W(n)^2}, \quad m\geq k.
$$
By symmetry, 
$$
\PP(S^{(k)}_{\tau}>\frac{1}{4} \sqrt{A_k }, \tau <\infty )=\PP(S^{(k)}_{\tau}<-\frac{1}{4} \sqrt{A_k }, \tau <\infty )= \frac{1}{2} \PP(\tau<\infty) \geq \frac{5}{16} ,
$$
and 
$$
\PP(S^{(k)} \geq  S^{(k)}_{\tau}| S^{(k)}_{\tau}>\frac{1}{4} \sqrt{A_k }, \tau <\infty )\geq  \frac{1}{2},
$$
These two inequalities imply the desired result.
\end{proof}

\begin{proof}[Proof of Theorem \ref{wnnewprop}] 
    By (\ref{decomsigman}), for any $j \in \{1,2,\cdots,N_c\}$ and $n\geq a_j+ d-1$, if $\sigma_{n+1}(j)<\infty$, we can write
$$
\sigma_{n+1}(j)-\sigma_{n}(j)=\sum_{i=0}^{d-1}\frac{b_i}{W(n-i)}, 
$$
where $b_i\geq 0$ and $\sum_{i=0}^{d-1} b_i=\xi^{(j)}_{n+1}$. Observe that 
$$
\frac{b_{d-1}}{W(n-d+1)}+\frac{b_{d-2}}{W(n-d+2)} =b_{d-1} \left(\frac{1}{W(n-d+1)}-\frac{1}{W(n-d+2)} \right)+ \frac{b_{d-1}+b_{d-2}}{W(n-d+2)}
$$
which belongs to the closed interval
$$
\left[\frac{b_{d-1}+b_{d-2}}{W(n-d+2)}-\delta, 
  \frac{b_{d-1}+b_{d-2}}{W(n-d+2)}+\delta\right]\ \text{where}\ \delta:=\left|\frac{\xi^{(j)}_{n+1}}{W(n-d+1)}-\frac{\xi^{(j)}_{n+1}}{W(n-d+2)}\right|.
$$
Repeating this procedure $d-1$ times gives 
    \begin{equation}
        \label{couplineq}
      |\sigma_{n+1}(j)-\sigma_{n}(j)- \frac{\xi^{(j)}_{n+1}}{W(n)}|\leq   \sum_{i=0}^{d-2} \xi^{(j)}_{n+1}\left|\frac{1}{W(n-i-1)}-\frac{1}{W(n-i)}\right|.
    \end{equation}
 
    \textbf{Case (i)}: We assume that (\ref{Wnseccond}) holds. Then, for any $k\geq n$, 
    \begin{equation}
        \label{wuwkbound}
          \frac{W(n)}{W(k)}-1 \leq W(n)\left|\frac{1}{W(k)}-\frac{1}{W(n)}\right| \leq C.
    \end{equation}
    In particular, $W(n)/\inf_{k\geq n}W(k) \leq C+1$ for any $n\geq 1$. As in (\ref{Znkddad}), the conditional Borel-Cantelli lemma then implies that a.s. there is an infinite sequence of finite stopping times $\{\tau_{n_kd}\}_{k\geq 1}$ such that at each time $\tau_{n_kd}$, there exists some $i_k \in \{1,2,\cdots,N_c\}$, 
    \begin{equation}
        \label{advantageiK}
        Z_{n_kd}(i_k) \geq d+ \max_{j\neq i_k}\{Z_{n_kd}(j)\}.
    \end{equation}
As in the proof of Theorem \ref{stronincfix}, at time $\tau_{n_kd}$, for any $j \in \{1,2,\cdots,N_c\}$, the remaining time of the timer on $e_j$ has an exponential distribution with rate $W(Z_{n_kd}(j))$, which, by a slight abuse of notation, we denote by $\xi^{(j)}_{Z_{n_kd}(j)+1}/W(Z_{n_kd}(j))$. We assume that $Z_{n_kd}(q)=\max_{j\neq i_k}\{Z_{n_kd}(j)\}$ for some $q\neq i_k$. Then,
\begin{equation}
    \label{condExiikq}
    \begin{aligned}
    &\quad \EE\left(\sum_{n=Z_{n_kd}(i_k)}^{\infty}(\xi^{(i_k)}_{n+1}+\xi^{(q)}_{n+1})\sum_{\ell=0}^{d-2}\left|\frac{1}{W(n-\ell-1)}-\frac{1}{W(n-\ell)}\right| \mid \FF_{\tau_{n_kd}} \right) \\
    &= \sum_{n=Z_{n_kd}(i_k)}^{\infty}\sum_{\ell=0}^{d-2}\left|\frac{2}{W(n-\ell-1)}-\frac{2}{W(n-\ell)}\right|  \\
    &\leq 2(d-1) \sum_{n=Z_{n_kd}(i_k)-d+1}^{\infty}\left|\frac{1}{W(n)}-\frac{1}{W(n+1)}\right|\leq \frac{2C(d-1) }{W(Z_{n_kd}(i_k)-d+1)},
  \end{aligned}
\end{equation}
where we used that for each $n\geq Z_{n_kd}(i_k)-d+1$, the term $|\frac{2}{W(n)}-\frac{2}{W(n+1)}|$ is counted at most $d-1$ times in the sum in the second line, and the last inequality follows from (\ref{Wnseccond}).

By Markov inequality, for any positive integrable random variable $\Theta$, if $m(\Theta)$ denotes its median value, then
$$
\frac{1}{2}\leq \PP(\Theta\geq m(\Theta)) \leq \frac{\EE \Theta}{m(\Theta)},
$$
and thus $m(\Theta) \leq 2 \EE \Theta$. In particular, by (\ref{condExiikq}), 
\begin{equation}
    \label{summedianbd}
   \PP\left(\sum_{n=Z_{n_kd}(i_k)}^{\infty}\sum_{\ell=0}^{d-2}\left|\frac{\xi^{(i_k)}_{n+1}+\xi^{(q)}_{n+1}}{W(n-\ell-1)}-\frac{\xi^{(i_k)}_{n+1}+\xi^{(q)}_{n+1}}{W(n-\ell)}\right| \leq \frac{4C(d-1)}{W(Z_{n_kd}(i_k)-d+1)} \mid \FF_{\tau_{n_kd}} \right) \geq \frac{1}{2}.
\end{equation}
 By (\ref{wuwkbound}) and (\ref{advantageiK}), starting from time $\tau_{n_kd}$, the time needed to visit $e_q$ once more is
\begin{equation}
    \label{Iqextra}
    \frac{\xi^{(q)}_{Z_{n_kd}(q)+1}}{W(Z_{n_kd}(q))} \geq \frac{\xi^{(q)}_{Z_{n_kd}(q)+1}}{(C+1)W(Z_{n_kd}(i_k)-d+1)}.
\end{equation}
Again, here $\xi^{(q)}_{Z_{n_kd}(q)+1}/W(Z_{n_kd}(q))$ should be interpreted as the remaining time of the timer on $e_q$, which is independent of $\{\xi^{(i_k)}_{n+1},\xi^{(q)}_{n+1}\}_{n\geq Z_{n_kd}(i_k)}$ in virtue of (\ref{advantageiK}). Therefore, by (\ref{summedianbd}) and that $\xi^{(q)}_{Z_{n_kd}(q)+1} \sim \operatorname{Exp}(1)$, 
$$
 \PP(E_{q,i_k}\mid \FF_{\tau_{n_kd}})\geq \frac{1}{2} \PP(\xi^{(q)}_{Z_{n_kd}(q)+1}>4C(d-1)(C+1)\mid \FF_{\tau_{n_kd}}) \geq \frac{1}{2}e^{-4C(d-1)(C+1)} ,  
$$
where the event $E_{q,i_k}$ is defined by 
$$
E_{q,i_k}:=\left\{\frac{\xi^{(q)}_{Z_{n_kd}(q)+1}}{(C+1)W(Z_{n_kd}(i_k)-d+1)}> \sum_{n=Z_{n_kd}(i_k)}^{\infty}\sum_{\ell=0}^{d-2}\left|\frac{\xi^{(i_k)}_{n+1}+\xi^{(q)}_{n+1}}{W(n-\ell-1)}-\frac{\xi^{(i_k)}_{n+1}+\xi^{(q)}_{n+1}}{W(n-\ell)}\right|\right\}.
$$
By symmetry (one may interchange $\{\xi^{(i_k)}_{n+1}\}_{n\geq Z_{n_kd}(i_k)}$ and $\{\xi^{(q)}_{n+1}\}_{n\geq Z_{n_kd}(i_k)}$), 
\begin{equation}
    \label{Eqlowbd}
   \begin{aligned}
   &\quad \PP(\{\sum_{n=Z_{n_kd}(i_k)}^{\infty}\frac{\xi^{(i_k)}_{n+1}}{W(n)}<\sum_{n=Z_{n_kd}(i_k)}^{\infty}\frac{\xi^{(q)}_{n+1}}{W(n)}\}\cap E_{q,i_k}\mid \FF_{\tau_{n_kd}}) \\
   &= \PP(\{\sum_{n=Z_{n_kd}(i_k)}^{\infty}\frac{\xi^{(q)}_{n+1}}{W(n)}<\sum_{n=Z_{n_kd}(i_k)}^{\infty}\frac{\xi^{(i_k)}_{n+1}}{W(n)}\}\cap E_{q,i_k}\mid \FF_{\tau_{n_kd}}) \\
   &=\frac{\PP( E_{q,i_k}\mid \FF_{\tau_{n_kd}})}{2} \geq  \frac{1}{4}e^{-4C(d-1)(C+1)}=:c_3.
\end{aligned}
\end{equation}
For $j\neq i_k$, we let $E_j$ be the event that
$$
\begin{aligned}
   &\frac{\xi^{(j)}_{Z_{n_kd}(j)+1}}{(C+1)W(Z_{n_kd}(i_k)-d+1)}+\sum_{n=Z_{n_kd}(i_k)}^{\infty}\left(\frac{\xi^{(j)}_{n+1}}{W(n)}-\sum_{\ell=0}^{d-2}\xi^{(j)}_{n+1}\left|\frac{1}{W(n-\ell-1)}-\frac{1}{W(n-\ell)}\right|\right) \\
   &> \sum_{n=Z_{n_kd}(i_k)}^{\infty}\left(\frac{\xi^{(i_k)}_{n+1}}{W(n)}+\sum_{\ell=0}^{d-2}\xi^{(i_k)}_{n+1}\left|\frac{1}{W(n-\ell-1)}-\frac{1}{W(n-\ell)}\right|\right).
\end{aligned}
$$
Note that by symmetry, (\ref{Eqlowbd}) still holds if one replaces $q$ by $j\neq i_k$. Thus, 
$\PP(E_j\mid \FF_{\tau_{n_kd}})=\PP(E_q\mid \FF_{\tau_{n_kd}})\geq c_3$. Since $\{\xi^{(i)}_n\}_{1\leq i\leq N_c, n\geq 1}$ are i.i.d., by H\"{o}lder's inequality, one has
$$
\begin{aligned}
    \PP(\bigcap_{j\neq i_k}E_j\mid \FF_{\tau_{n_kd}})&=\EE \left[ \PP\left(\bigcap_{j\neq i_k}E_j\mid  \{\xi^{(i_k)}_{n+1}\}_{n\geq Z_{n_kd}(i_k)}, \FF_{\tau_{n_kd}}\right)\mid \FF_{\tau_{n_kd}}\right] \\
    &= \EE \left[\PP\left(E_{q}\mid \{\xi^{(i_k)}_{n+1}\}_{n\geq Z_{n_kd}(i_k)}, \FF_{\tau_{n_kd}}\right)^{N_c-1} \mid \FF_{\tau_{n_kd}}\right] \\
    &\geq \PP(E_{q}\mid \FF_{\tau_{n_kd}})^{N_c-1}.
\end{aligned}
$$
In virtue of (\ref{couplineq}), on $\bigcap_{j\neq i_k}E_j$, we have $\lim_{n\to \infty}\sigma_n(i_k)<\infty$ and 
$$
\lim_{n\to \infty}\sigma_n(i_k)-\tau_{n_kd}<  \min_{j\neq i_k}\{ \lim_{n\to \infty}\sigma_n(j)-\tau_{n_kd} \}.
$$
That is, the remaining time needed to visit $e_{i_k}$ i.o. is strictly less than that needed for any other edge. In other words, by Proposition \ref{ZkdNin}, only balls of color $i_k$ are taken infinitely often. Therefore,
   $$
\PP_1^W(\mathcal{M}\mid \FF_{\tau_{n_kd}}) \geq \PP(\bigcap_{j\neq i_k}E_j\mid \FF_{\tau_{n_kd}})  \geq c_3^{N_c-1}>0.
   $$
We conclude that $\PP_1^W(\mathcal{M})=1$ by Levy's 0-1 law.

\textbf{Case (ii)}: We assume that (\ref{Wn2Anlarge}) holds. By a slight abuse of notation, for $k\geq 1$, we let $i_k$ be such that $Z_{kd}(i_k) = \max_{1\leq j\neq N_c}\{Z_{kd}(j)\}$. For any $j\neq i_k$, by Lemma \ref{martdeviation}, 
\begin{equation}
    \label{jiklemmart}
    \PP(\sum_{n=Z_{kd}(i_k)}^{\infty}\frac{\xi^{(j)}_{n+1}-\xi^{(i_k)}_{n+1}}{W(n)} > \frac{\sqrt{A_{Z_{kd}(i_k)}}}{4} \mid \FF_{\tau_{kd}})\geq \frac{5}{32},
\end{equation}
and, by Markov's inequality and (\ref{condExiikq}), for large $k$,
\begin{equation}
    \label{markovikj}
    \begin{aligned}
        &\quad\PP\left(\sum_{n=Z_{kd}(i_k)}^{\infty}\sum_{\ell=0}^{d-2}\left|\frac{\xi^{(i_k)}_{n+1}+\xi^{(j)}_{n+1}}{W(n-\ell-1)}-\frac{\xi^{(i_k)}_{n+1}+\xi^{(j)}_{n+1}}{W(n-\ell)}\right| \geq  \frac{\sqrt{A_{Z_{kd}(i_k)}}}{4}  \mid \FF_{\tau_{kd}} \right) \\
        &\leq \frac{8(d-1)\delta_{Z_{kd}(i_k)-d+1}}{\sqrt{A_{Z_{kd}(i_k)}}}=\frac{8(d-1)\delta_{Z_{kd}(i_k)-d+1}}{\sqrt{A_{Z_{kd}(i_k)-d+1}}}\frac{\sqrt{A_{Z_{kd}(i_k)-d+1}}}{\sqrt{A_{Z_{kd}(i_k)}}} \leq \frac{1}{8}.
    \end{aligned}
\end{equation}
where we used that $A_{n+1}/A_n$ converges to $1$ by (\ref{Wn2Anlarge}). By a slight abuse of notation, we let $E_j$ be the event that
$$
\sum_{n=Z_{kd}(i_k)}^{\infty}\frac{\xi^{(j)}_{n+1}-\xi^{(i_k)}_{n+1}}{W(n)} > \sum_{n=Z_{kd}(i_k)}^{\infty}\left((\xi^{(j)}_{n+1}+\xi^{(i_k)}_{n+1})\sum_{\ell=0}^{d-2}\left|\frac{1}{W(n-\ell-1)}-\frac{1}{W(n-\ell)}\right|\right).
$$
We deduce from (\ref{jiklemmart}) and (\ref{markovikj}) that 
$$
\PP(E_j \mid \FF_{\tau_{kd}}) \geq \frac{5}{32}-\frac{1}{8}=\frac{1}{32}. 
$$
The rest of the proof follows the same lines as that of Case 1: H\"{o}lder's inequality and (\ref{couplineq}) imply that for all large $k$,
$$
\PP_1^W(\mathcal{M}\mid \FF_{\tau_{kd}}) \geq \PP(\bigcap_{j\neq i_k}E_j\mid \FF_{\tau_{kd}})  \geq (\frac{1}{32})^{N_c-1},
   $$
   which shows that $\PP_1^W(\mathcal{M})=1$ by Levy's 0-1 law.
\end{proof}

\section{Coupling} 
\label{seccouple}
\begin{proof}[Proof of Lemma \ref{couplcomparenewurn}]
    Let $\{U_n^{(i)}\}_{n\geq 1,1\leq i\leq 2}$ be i.i.d. uniform random variables on $(0, 1)$. For any $n\in \NN$, we set $\tilde{B}_{2n+1}(1)=\tilde{B}_{2n}(1)+1$, resp. $B_{n+1}(1)=B_n(1)+1$ if 
    $$
 U_{n+1}^{(1)}< \frac{W(\tilde{B}_{2n}(1))}{W(\tilde{B}_{2n}(1))+W(\tilde{R}_{2n}^{*})}, \ \text{resp.}\ \  U_{n+1}^{(1)}< \frac{pW(B_n^{*})}{W(B_n^{*})+W(R_n^{*})}+ \frac{(1-p)W(B_n(1))}{W(B_n(1))+W(R_n(1))},
    $$
    otherwise, we set $\tilde{R}_{2n+1}(1)=\tilde{R}_{2n}(1)+1$, resp. $R_{n+1}(1)=R_n(1)+1$; we set $\tilde{B}_{2n+2}(2)=\tilde{B}_{2n}(2)+1$, resp. $B_{n+1}(2)=B_n(2)+1$ if 
    $$
 U_{n+1}^{(2)}< \frac{W(\tilde{B}_{2n}(2))}{W(\tilde{B}_{2n}(2))+W(\tilde{R}_{2n+1}^{*})}, \ \text{resp.}\ \   U_{n+1}^{(2)}< \frac{pW(B_n^{*})}{W(B_n^{*})+W(R_n^{*})}+ \frac{(1-p)W(B_n(2))}{W(B_n(2))+W(R_n(2))},
    $$
    otherwise, we set $\tilde{R}_{2n+2}(1)=\tilde{R}_{2n}(2)+1$, resp. $R_{n+1}(2)=R_n(2)+1$. Then it is easy to check that $\left(B_n, R_n\right)_{n \in \mathbb{N}}$ and $\left(\tilde{B}_n, \tilde{R}_n\right)_{n \in \mathbb{N}}$ defined above have the desired laws. One can then prove (\ref{coupleine}) by induction.
\end{proof}

\begin{proof}[Proof of Lemma \ref{estnewurntimelines}]
    As in the proof of Theorem \ref{wnnewprop}, we use a time-lines construction to prove Lemma \ref{estnewurntimelines}. Let $G=(V,E)$ be a directed multigraph with 
    $$V=\{v_1,v_2\}, \quad E=\{(v_1,b,v_2),(v_2,b,v_1),(v_1,r,v_2),(v_2,r,v_1)\}.$$ 
    where $(v_1,b,v_2)$ and $(v_2,b,v_1)$ are two arcs from $v_1$ to $v_2$ and $v_2$ to $v_1$, respectively. We regard $e_r:=\{(v_1,r,v_2), (v_2,r,v_1)\}$ as an undirected edge. 
    \begin{center}
        \begin{tikzpicture}
            [ultra thick]
        \draw (-4, 0.5) node[circle, draw](x) {$v_1$};
        \draw ( 0.5, 0.5) node[circle, draw](y) {$v_2$};
        \draw[black, ->] (x) to[bend right=30] node[above]{$(v_1,b,v_2)$} (y);
        \draw[red,<->] (y) to[bend right=50] node[above]{$e_r$} (x);
        \draw[black,->] (y) to[bend left=60] node[below]{$(v_2,b,v_1)$} (x);
        \end{tikzpicture}
    \end{center}
    
    Let $\{\xi^{(r)}_n\}_{n\geq 1}$, $\{\xi^{(v_1,b)}_n\}_{n\geq 1}$, $\{\xi^{(v_2,b)}_n\}_{n\geq 1}$ be independent Exp(1)-distributed random variables. We define a continuous-time jump process $Y=(Y_t)_{t\geq 0}$ on $G$:\\
    (i) Define, on each (directed or undirected) edge $e \in \{(v_1,b,v_2), (v_2,b,v_1),e_r\}$, independent point processes (alarm times) $\{ V_{n}^{(e)}\}_{n\geq 1}$: for each $n\in \NN$,
\begin{equation}
    \label{Vclockrb}
    V_{n+1}^{(v_1,b,v_2)}:=\sum_{j=\tilde{B}_0(2)}^{\tilde{B}_0(2)+n} \frac{\xi^{(v_2,b)}_j}{W(j)}, \quad V_{n+1}^{(v_2,b,v_1)}:=\sum_{j=\tilde{B}_0(1)}^{\tilde{B}_0(1)+n} \frac{\xi^{(v_1,b)}_j}{W(j)},\quad V_{n+1}^{(e_r)}:=\sum_{j=\tilde{R}_0^{*}}^{\tilde{R}_0^{*}+n} \frac{\xi^{(r)}_j}{W(j)}.
\end{equation}
 (ii) Each edge $e$ has its own clock, denoted by $\tilde{T}_{e}(t)$. If $e=(v_1,b,v_2)$ (resp. $(v_2,b,v_1)$), $\tilde{T}_{e}$ runs when $Y$ is at $v_1$ (resp. $v_2$). For $e=e_r$, set $\tilde{T}_{e_r}(t):=t$. \\
(iii) Set $Y_0:=v_2$. If at time $t>0$, the clock of an edge $e$ rings, i.e. $\tilde{T}_{e}(t)=V_{k}^{(e)}$ for some $k>0$, then $Y$ jumps to cross $e$ instantaneously.

Let $0=\tau_0<\tau_1<\tau_2<\cdots$ be the jumping times of $Y$. For $i=1,2$, as in (\ref{defZni}), let $Z_n^{(b)}(i)$ be the number of visits to $(v_{3-i},b,v_{i})$ up to time $\tau_n$ plus $\tilde{B}_0(i)$, and let $Z_n^{(r)}(i)$ be the number of visits to $(v_{3-i},r,v_{i})$ up to time $\tau_n$ plus $\tilde{R}_0(i)$ (note that here we distinguish $(v_1,r,v_2)$ and $(v_2,r,v_1)$). Then, as in Proposition \ref{ZkdNin}, one can show by the memoryless property of exponentials that 
    $$
    (Z_{n}^{(b)}(i),Z_{n}^{(r)}(i))_{1\leq i \leq 2, n \in \NN} \stackrel{\mathcal{L}}{=} (\tilde{B}_{n}(i),\tilde{R}_{n}(i))_{1\leq i \leq 2, n\in \NN}.
    $$
We may assume that $(Z_{n}^{(b)}(i),Z_{n}^{(r)}(i))_{1\leq i \leq 2, n \in \NN} = (\tilde{B}_n(i),\tilde{R}_n(i))_{1\leq i \leq 2, n\in \NN}$ on some probability space. We denote by $(\tilde{\FF}_{t})$ the natural filtration of $Y$. 

Now, assume that $\tilde{R}_{2n}^{*}<\varepsilon_1n$, and in particular, $\tilde{B}_{2n}(i)\geq (1-\varepsilon_1)n$ for $i=1,2$. Starting from time $\tau_{2n}$, the total time that $Y$ needs to spend to cross the undirected edge $e_r$ infinitely often is 
    \begin{equation}
     \label{totaltimelower}
     \sum_{j=\tilde{R}_{2n}^{*}}^{\infty} \frac{\xi_{j+1}^{(r)}}{W(j)}\geq \sum_{j=\lceil \varepsilon_1 n \rceil}^{\infty} \frac{\xi_{j+1}^{(r)}}{W(j)}=: T_{n}^{(r)},
    \end{equation}
    where $\lceil \cdot\rceil$ is the usual ceiling function. Again, the first term $\xi_{\tilde{R}_{2n}^{*}+1}^{(r)}/(\tilde{R}_{2n}^{*})^{\alpha}$ should be interpreted as the remaining time of the clock on $e_r$ at time $\tau_{2n}$. On the other hand, the total time that $Y$ needs to spend to cross both $(v_1,b,v_2)$ and $(v_2,b,v_1)$ infinitely often is 
    \begin{equation}
     \label{totaltimeupper}
     T_{n}^{(b)}:=\sum_{j=\tilde{B}_{2n}(1)}^{\infty} \frac{\xi_{j+1}^{(v_1,b)}}{W(j)} +\sum_{j=\tilde{B}_{2n}(2)}^{\infty} \frac{\xi_{j+1}^{(v_2,b)}}{W(j)}.
 \end{equation}
 Note that up to time $\tau_{\infty}$, the time $Y$ spends at $v_1$, resp. $v_2$, is upper bounded by $\sum_{j=\tilde{B}_{2n}(2)}^{\infty} \xi_{j+1}^{(v_2,b)}/W(j) $, resp. $\sum_{j=\tilde{B}_{2n}(1)}^{\infty} \xi_{j+1}^{(v_1,b)}/W(j)$. Now using properties of exponential random variables, we have, by (\ref{totaltimelower}) and (\ref{totaltimeupper}),
$$
\operatorname{Var}(\widehat{T}_{n}^{(r)}\mid \tilde{\FF}_{\tau_{2n}})=\sum_{j=\lceil \varepsilon_1 n \rceil+1}^{\infty} \frac{1}{W(j)^2},\ \text{ where}\ \widehat{T}_{n}^{(r)}:=T_{n}^{(r)}-\frac{\xi_{\lceil \varepsilon_1 n \rceil+1}^{(r)}}{W(\lceil \varepsilon_1 n \rceil)},
$$
and  
$$
\EE (\widehat{T}_{n}^{(r)} \mid \tilde{\FF}_{\tau_{2n}}) = \sum_{j=\lceil \varepsilon_1 n \rceil + 1}^{\infty} \frac{1}{W(j)}, \quad \EE (T_{n}^{(b)} \mid \tilde{\FF}_{\tau_{2n}})\leq  \sum_{j=\lceil (1-\varepsilon_1) n \rceil}^{\infty} \frac{2}{W(j)}.
$$
By Chebyshev's inequality, 
\begin{equation}
    \label{Tnrine}
    \begin{aligned}
        &\quad \PP(T_{n}^{(r)}>\frac{1}{4}\sum_{j=\lceil \varepsilon_1 n \rceil}^{\infty} \frac{1}{W(j)} \mid \tilde{\FF}_{\tau_{2n}})\\
        &\geq \PP\left( \xi_{\lceil \varepsilon_1 n \rceil+1}^{(r)}\geq 1, \widehat{T}_{n}^{(r)} >-\frac{3}{4W(\lceil \varepsilon_1 n \rceil)}+ \frac{1}{4}\sum_{j=\lceil \varepsilon_1 n \rceil+1}^{\infty} \frac{1}{W(j)}  \mid \tilde{\FF}_{\tau_{2n}}\right) \\
        &\geq \frac{1}{e}\left(1- \PP\left(\EE (\widehat{T}_{n}^{(r)} \mid \tilde{\FF}_{\tau_{2n}}) -\widehat{T}_{n}^{(r)}\geq \frac{3}{4}\sum_{j=\lceil \varepsilon_1 n \rceil}^{\infty} \frac{1}{W(j)}\mid \tilde{\FF}_{\tau_{2n}}\right)\right)\\
        &\geq \frac{1}{e}\left(1-\frac{16}{9}\left(\sum_{j=\lceil \varepsilon_1 n \rceil+1}^{\infty} \frac{1}{W(j)^2}\right)\left( \sum_{j=\lceil \varepsilon_1 n \rceil}^{\infty} \frac{1}{W(j)} \right)^{-2}\right) \geq \frac{1}{9e},
        \end{aligned}
\end{equation}
where we used the monotonicity of $W$ to get 
$$
\left( \sum_{j=\lceil \varepsilon_1 n \rceil}^{\infty} \frac{1}{W(j)} \right)^{2} \geq \sum_{j=\lceil \varepsilon_1 n \rceil+1}^{\infty} \left(\frac{1}{W(\lceil \varepsilon_1 n \rceil)} +\frac{1}{W(j)} \right)\frac{1}{W(j)} \geq 2\sum_{j=\lceil \varepsilon_1 n \rceil+1}^{\infty} \frac{1}{W(j)^2}.
$$
On the other hand, Markov's inequality implies that
\begin{equation}
    \label{ChebyshevT}
    \PP\left(T_n^{(b)}\geq \frac{1}{4}\sum_{j=\lceil \varepsilon_1 n \rceil}^{\infty} \frac{1}{W(j)}\mid \tilde{\FF}_{\tau_{2n}}\right)\leq 8\left(\sum_{j=\lceil (1-\varepsilon_1) n \rceil}^{\infty} \frac{1}{W(j)}\right)\left(\sum_{j=\lceil \varepsilon_1 n \rceil}^{\infty} \frac{1}{W(j)}\right)^{-1},
\end{equation}
In virtue of (\ref{remKcond}), the right-hand side of (\ref{ChebyshevT}) can be made arbitrarily small for all $n\geq \kappa$ by first choosing a small $\varepsilon_1$ and then choosing a large $\kappa$. Using (\ref{Tnrine}) and (\ref{ChebyshevT}), by possibly choosing a smaller $\varepsilon_1$ and a larger $\kappa$, we have
$$
\PP(T_n^{(b)} < T_n^{(r)}\mid \tilde{\FF}_{\tau_{2n}}) \geq \PP(T_n^{(b)} <\frac{1}{4}\sum_{j=\lceil \varepsilon_1 n \rceil}^{\infty} \frac{1}{W(j)}< T_n^{(r)}\mid \tilde{\FF}_{\tau_{2n}})> \frac{1}{10e}
$$
if $n\geq \kappa$ and $\tilde{R}_{2n}^{*}<\varepsilon_1n$. It remains to observe that on the event $\{T_n^{(b)} < T_n^{(r)}\}$, only black edges are crossed infinitely often, that is, only black balls are drawn infinitely often.
\end{proof}

\section{Some open questions}

For the interacting urn mechanism with strong reinforcement, some interesting questions remain unsolved.

\begin{enumerate}[label=(\roman*)]
    \item For power function/polynomial reinforcements, an important question is whether there is a phase transition at $p_{\alpha}$. In Remark \ref{palpharem}, we conjecture that $\tilde{p}_{\alpha}= p_{\alpha}$. If this is true, is $p_{\alpha}$ increasing in $\alpha$? (Intuitively speaking, the reinforcement becomes stronger as $\alpha$ grows.) And does the limit of $p_{\alpha}/(\alpha-1)$ exist as $\alpha$ approaches 1 from above?  To solve these questions, we may need a better understanding of the system (\ref{odedxFm}), or we need to couple $\PP_{p_1}^{(\alpha)}$ and $\PP_{p_2}^{(\alpha)}$ for $p_1<p_2$.

    \item We conjecture that for a large class of strong reinforcement sequences, say, $\{W(n)\}_{n\in \NN}$ is non-decreasing and strong, one has $\PP_p^W(\mathcal{M})=1$ if $p>1/2$. Currently, this assertion is proved only for exponential and polynomial reinforcements, see \cite[Theorem 3.2]{launay2011interacting} and Theorem \ref{pmlarger12}.
    \item As was conjectured in \cite{launay2012urns}, can one prove that $\PP_1^W(\mathcal{M})=1$ if one only assumes that $\{W(n)\}_{n\in \NN}$ is a strong reinforcement sequence? The best result in this direction, known to date, seems to be Theorem \ref{wnnewprop}.
\end{enumerate}

\section{Acknowledgement}

I am very grateful to Professor Tarrès, my Ph.D. advisor, for inspiring the choice of this subject.

\bibliographystyle{plain}
\bibliography{math_ref}

\end{document}